\documentclass[11pt]{amsart}

\usepackage{amssymb,amsthm}

\usepackage[hyphens]{url}
\usepackage{comment}
\usepackage{graphicx}
\usepackage{caption}
\usepackage{subcaption}
\usepackage{amsrefs}
\usepackage{xcolor}

\newtheorem{theorem}{Theorem}
\newtheorem{corollary}{Corollary}
\newtheorem{proposition}{Proposition}
\newtheorem{lemma}{Lemma}

\theoremstyle{definition}
\newtheorem{definition}{Definition}

\theoremstyle{remark}
\newtheorem{remark}{Remark}

\theoremstyle{definition}

\newcommand{\Z}{\mathbb{Z}}
\newcommand{\C}{\mathbb{C}}
\newcommand{\Q}{\mathbb{Q}}

\newcommand{\T}{\mathbb{T}}
\newcommand{\R}{\mathbb{R}}
\newcommand{\B}{B}
\newcommand{\mc}{\mathcal}
\newcommand{\mb}{\mathbf}
\newcommand{\N}{\mathsf N}

\newcommand{\ZZ}{\mathsf Z}

\newcommand{\V}{\mathsf V}
\newcommand{\mbhat}[1]{\hat{\mb #1}}
\newcommand{\mbbar}[1]{\bar{\mb #1}}

\DeclareMathOperator{\supp}{supp}
\DeclareMathOperator{\Gal}{Gal}
\DeclareMathOperator{\Sp}{Sp}

\usepackage{hyperref}

\title{Fractal uncertainty for discrete 2D Cantor sets}
\author{Alex Cohen}
\thanks{Supported by a Fannie and John Hertz Foundation fellowship and an NSF GRFP fellowship. Partial support from NSF CAREER grant DMS-1749858 is acknowledged.}
\date{June 2022}

\begin{document}
\begin{abstract}
We prove that a self-similar Cantor set in $\mathbb{Z}_N \times \mathbb{Z}_N$ has a fractal uncertainty principle if and only if it does not contain a pair of orthogonal lines. The key ingredient in our proof is a quantitative form of Lang's conjecture in number theory due to Ruppert and Beukers \& Smyth. Our theorem answers a question of Dyatlov and has applications to open quantum maps. 
\end{abstract}

\maketitle
\section{Introduction}

\subsection{One dimensional fractal uncertainty}
The Bourgain-Dyatlov \cite{BourgainDyatlov} fractal uncertainty principle (FUP) says, in a precise quantitative sense, that a function $f: \R \to \C$ cannot simultaneously have large $L^2$ mass on a fractal set in physical space and large $L^2$ mass on a fractal set in Fourier space. This theorem and its variants have many applications to quantum chaos, see the survey article of Dyatlov \cite{DyatlovFUPSurvey}. The proof of FUP in \cite{BourgainDyatlov} is quite tricky, but the analagous result in the discrete setting has similar ingredients and is much simpler. 

First some notation. In this paper $\Z_N = \Z/N\Z$ refers to the integers mod $N$. 
We use the unitary discrete Fourier transform $\mc F: \ell^2(\Z_N) \to \ell^2(\Z_N)$, given by 
\begin{align*}
	\mc F f(\xi) &= \hat f(\xi) = \frac{1}{\sqrt{N}}\sum_{x \in \Z_N} f(x)e^{-\frac{2\pi i}{N}\xi x} \\ 
	\mc F^{-1} f(x) &= f^{\vee}(x) = \frac{1}{\sqrt{N}}\sum_{\xi \in \Z_N} f(\xi)e^{\frac{2\pi i}{N}\xi x}.
\end{align*}
We will also use the one dimensional and two dimensional tori $\T$, $\T^2$, which are identified as sets with $[0,1)$ and $[0, 1) \times [0,1)$.

Let us restrict attention to self-similar Cantor sets (when we say Cantor set we always mean self-similar). Fix a base $M$ and an alphabet $\mc A \subsetneq \Z_M$. Then let 
\begin{equation*}
	\mc X_k = \{a_0 + a_1M + \dots + a_{k-1}M^{k-1} \, :\, a_j \in \mc A\} \subset \Z_{M^k}
\end{equation*}
be the $k$th iterate. It will be convenient to let $N = M^k$, so $\mc X_k \subset \Z_N$. We say $\mc A$ has dimension $\delta_A = \log_M |\mc A|$, so $M^{k\delta} = |\mc X_k|$ for all $k$. Similarly, let $\mc Y_k$ be the Cantor iterates associated with the alphabet $\mc B$. 
Dyatlov \& Jin \cite{DyatlovJinBakersMaps} proved the following fractal uncertainty principle for discrete 1D Cantor sets. They were motivated by applications to open quantum maps, see \S\ref{sec:application_quantum_chaos} for more discussion.
\begin{theorem}[1D FUP \cite{DyatlovJinBakersMaps}*{Theorem 2}]\label{thm:1D_FUP}
For all alphabets $\mc A, \mc B \subsetneq \Z_M$, the estimate
\begin{equation}\label{eq:disc_fup}
	\|1_{\mc Y_k} \mc F 1_{\mc X_k}\|_{2 \to 2} \lesssim M^{-k\beta}
\end{equation}
holds for some $\beta > 0$. 
\end{theorem}

Because of self-similarity, Cantor sets enjoy the following submultiplicativity estimate (see \S\ref{sec:submul_proof} for a proof)
\begin{equation}\label{eq:submul}
	\|1_{\mc Y_{r+k}} \mc F 1_{\mc X_{r+k}}\|_{2 \to 2} \leq \|1_{\mc Y_r} \mc F 1_{\mc X_r}\|_{2 \to 2} \|1_{\mc Y_k} \mc F 1_{\mc X_k}\|_{2 \to 2} 
\end{equation}
which reduces (\ref{eq:disc_fup}) to the problem of proving that for \textit{some} $k > 0$, 
\begin{equation*}
	\|1_{\mc Y_k} \mc F 1_{\mc X_k}\|_{2 \to 2} < 1.
\end{equation*}
This estimate holds if there is no nonzero function $f$ with $\supp f \subset \mc X_k$ and $\supp \hat f \subset \mc Y_k$. To recap, proving a 1D FUP reduces to showing that for some $k$, there is no function $f$ with $\supp f \subset \mc X_k$ and $\supp \hat f \subset \mc Y_k$. 

In the general case of arbitrary porous sets (not necessarily Cantor sets), submultiplicativiy is replaced by an induction on scales argument which allows one to find significant $L^2$ mass of $\hat f$ in the gaps of $\mc Y_k$ at every scale.

\subsection{Two dimensional fractal uncertainty}\label{sec:2d_fup_sec_intro}

In two dimensions, Cantor sets are determined by an alphabet $\mc A \subsetneq \Z_M^2$. We set 
\begin{equation}\label{eq:X_k_cantor_defn}
	\mc X_k = \{(a_0 + \dots + a_{k}M^{k-1}, b_0 + \dots + b_kM^{k-1})\, :\, (a_j, b_j) \in \mc A\} \subset \Z_{N}^2,
\end{equation}
where $N:= M^k$.
We have $|\mc A| = M^{\delta}$ with $0 < \delta < 2$, and $|\mc X_k| = M^{k\delta}$. The unitary Fourier transform in two dimensions is given by 
\begin{align*}
	\mc F f(\xi, \eta) = \hat f (\xi, \eta) &= \frac{1}{N} \sum_{(x,y) \in \Z_N^2} f(x,y)e^{-\frac{2\pi i}{N} (x\xi + y\eta)} \\ 
	\hat f(\mb \xi) &= \frac{1}{N} \sum_{\mb x \in \Z_N^2} f(\mb x) e^{-\frac{2\pi i}{N} \mb x \cdot \mb \xi}\quad \text{ in vector notation.}
\end{align*}
Submultiplicativity (\ref{eq:submul}) holds in two dimensions as well (see \S\ref{sec:submul_proof}), so proving a 2D FUP reduces to showing that for some $k$, there is no nonzero $f$ with $\supp f \subset \mc X_k$ and $\supp \hat f \subset \mc Y_k$. 

Unfortunately, this claim is not true in general. Indeed, 
\begin{equation*}
	f(x,y) = N^{-1/2}1_{y = 0} \text{ has } \hat f = N^{-1/2}1_{x = 0}
\end{equation*}
and fractal sets can contain vertical and horizontal lines. We show that the fractal sets generated by the alphabets $\mc A, \mc B$ containing a pair of orthogonal lines is the only obstruction to a 2D FUP. 
For $\mc A \subset \Z_M^2$ an alphabet, let 
\begin{align*}
	\mb A = \overline{\{(x,y) \in \T^2\, :\, (\lfloor Mx\rfloor, \lfloor My\rfloor) \in \mc A\}}.
\end{align*}
This is a closed drawing of $\mc A$ in $\T^2$, and we draw the Cantor iterate $\mc X_k$ as 
\begin{equation}\label{eq:kth_iterate_drawing}
	\mb X_k = \overline{\{(x,y) \in \T^2\, :\, (\lfloor M^kx\rfloor, \lfloor M^ky\rfloor) \in \mc X_k\}} \subset \T^2.
\end{equation}
We write $\mb X = \bigcap_k \mb X_k \subset \T^2$ as the limiting Cantor set, so
\begin{align*}
	\mb A &= \mb X_0 \supset \mb X_1 \supset \mb X_2 \supset \cdots \supset \mb X, \\
	\mb X &= \{(0.a_0a_1\dots, 0.b_0b_1\dots)\, :\, (a_j, b_j) \in \mc A \text{ for all } j \geq 0\} \text{ in base } M.
\end{align*}
Note that if $x \in \T$ is of the form $a/M^k$ then there are two possible decimal expansions---the point $(x, y) \in \T$ is in $\mb X$ if \textit{some} decimal expansion has all digits in the alphabet. 
For $\mc B$ a second alphabet we write $\mb B \subset \T^2$ as the drawing of $\mc B$ and $\mb Y \subset \T^2$ as the limiting Cantor set for $\mc B$. 
We need these closed sets to state the condition of our main theorem. 

\begin{theorem}[2D FUP]\label{thm:2D_FUP}
Suppose $\mc A, \mc B$ are alphabets. Then either 
\begin{equation}\label{eq:has_orthogonal_lines}
	\R \mb v + \mb p \subset \mb X \quad \text{and}\quad \R \mb v^{\perp} + \mb q \subset \mb Y
\end{equation}
for some $\mb v = (a,b) \in \R^2 - \{0\}$, $\mb p, \mb q \in \T^2$,
or if not then $\mc X_k$, $\mc Y_k$ satisfy
\begin{equation}\label{eq:2D_FUP_formula}
	\|1_{\mc Y_k} \mc F 1_{\mc X_k}\|_{2 \to 2} \lesssim M^{-k\beta}
\end{equation}
for some $\beta > 0$. 
\end{theorem}
In particular, if $\mb X$ does not contain any line then it has a FUP. We note that in this theorem, $(a,b)$ can be taken to be integers. Otherwise $a/b$ is irrational and the coset $\R \mb v + \mb p$ is dense in $\T^2$, so it cannot lie entirely in the closed set $\mb X \subsetneq \T^2$. The main outside ingredient we use is Theorem \ref{thm:beukers_smyth} due to Ruppert \cite{Ruppert}*{Corollary 5} and Beukers \& Smyth \cite{BeukersSmyth}*{Theorem 4}, see \S\ref{sec:main_lemma}. 

In \S\ref{sec:thm_sharp} we show that this theorem is sharp: if $\mb X$, $\mb Y$ contain a pair of orthogonal lines, FUP will fail. Notice that the condition of the theorem depends on the limiting Cantor sets $\mb X, \mb Y \subset \T^2$, and it is not immediately clear when alphabets $\mc A, \mc B$ generate Cantor sets satisfying this orthogonal line condition. The following proposition reduces this question to a finite combinatorial problem. 

\begin{proposition}\label{prop:lines_in_cantor_set}
A line $\R \mb v + \mb p$ lies on $\mb X$ if and only if $\R \mb v + M^k \mb p$ lies on $\mb A$ for all $k \geq 0$. Additionally, suppose $(a,b) \in \Z^2 - \{0\}$ is given, $a, b$ coprime. In order for there to be some $\mb p$ with $\R \mb v + \mb p \subset \mb X$, we must have $\max(|a|, |b|) \leq M$. 
\end{proposition}
Proposition \ref{prop:lines_in_cantor_set} leaves open a natural algorithmic question. Given an alphabet $\mc A$ and vector $\mb v \in \Z^2 - \{0\}$, does there exist a point $\mb p \in \T^2$ such that $\R \mb v + \mb p \subset \mb X$? An efficient algorithm for this problem would lead to an efficient algorithm for testing when two alphabets $\mc A, \mc B$ satisfy the conditions of Theorem \ref{thm:2D_FUP}. For the proof and more discussion see \S\ref{sec:pf_lines_in_cantor_prop}.

\begin{remark}
Theorem \ref{thm:2D_FUP} refines Dyatlov's Conjecture 6.7 from \cite{DyatlovFUPSurvey}. That conjecture recognizes the potential obstruction of $\mb X, \mb Y$ containing a pair of vertical / horizontal or diagonal / antidiagonal lines (the case $\max(|a|, |b|) \leq 1$ in Proposition \ref{prop:lines_in_cantor_set}), but does account for lines with other slopes, which may occur in practice. See Figure \ref{fig:cantor_with_line}.
\end{remark}

\begin{figure}
\centering
\begin{subfigure}[t]{.5\textwidth}
  \centering
  \includegraphics[width=0.9\linewidth]{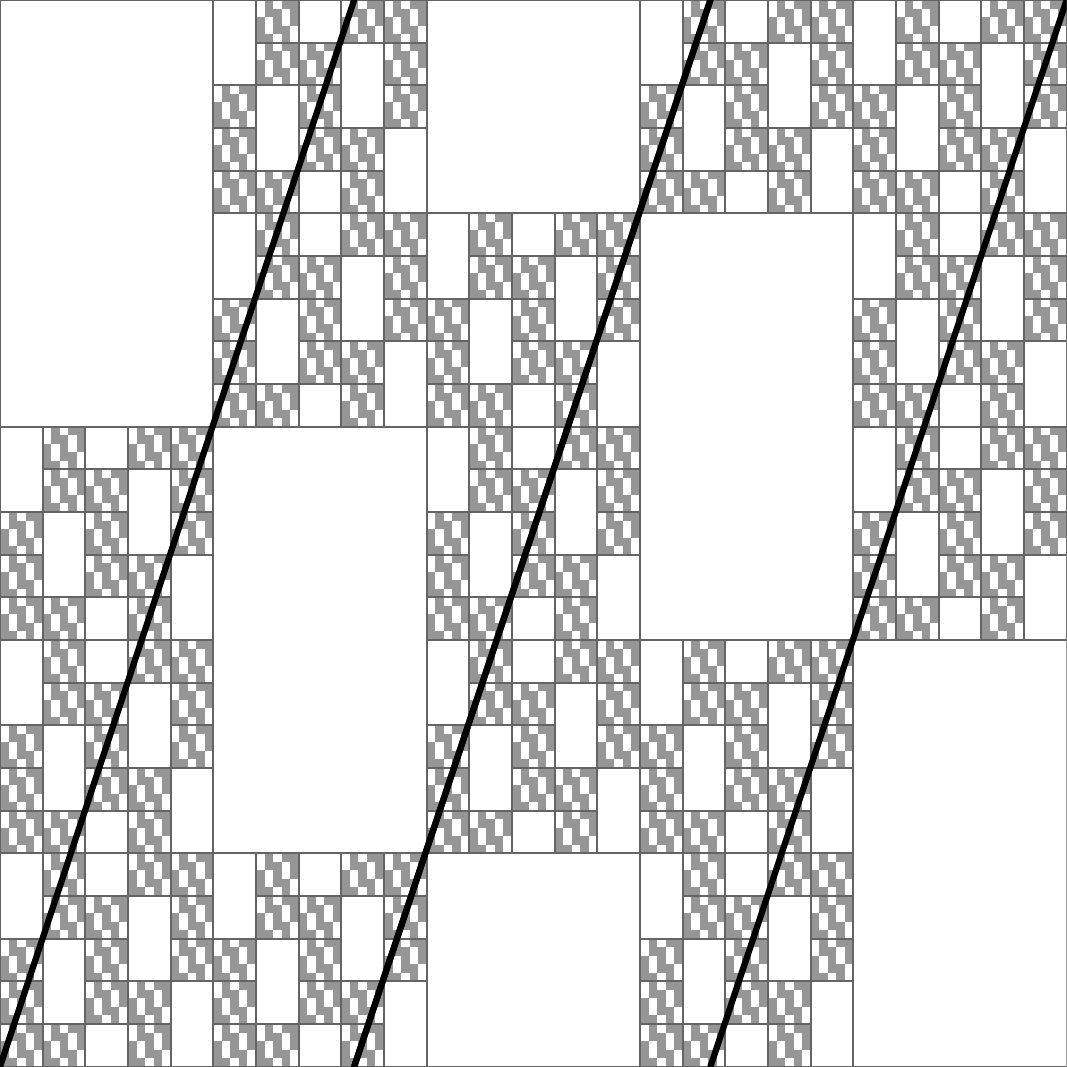}
  \caption{All iterates contain a skewed line}
\end{subfigure}%
\begin{subfigure}[t]{.5\textwidth}
  \centering
  \includegraphics[width=0.9\linewidth]{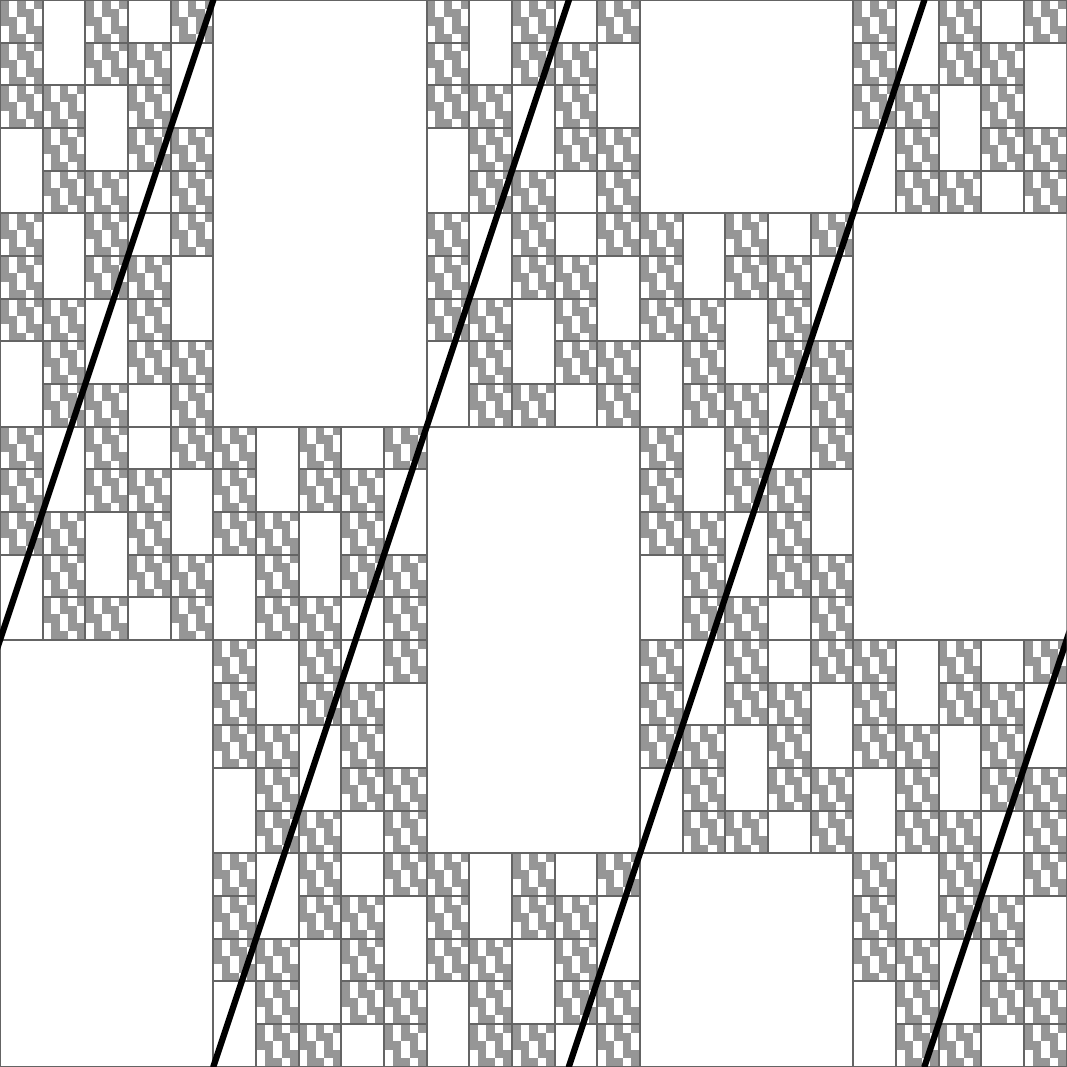}
  \caption{The first iterate contains a skewed line, but further iterates do not}
\end{subfigure}
\caption{Cantor sets can contain lines that aren't horizontal, vertical, or diagonal, but they are less stable}
\label{fig:cantor_with_line}
\end{figure}

Theorem \ref{thm:2D_FUP} is only interesting when $\frac{\delta_{A} +\delta_{B}}{2} \geq 1$. Indeed, Equation 6.8 from \cite{DyatlovFUPSurvey} says that (\ref{eq:2D_FUP_formula}) always holds with $\beta = \max\Bigl(0, 1 - \frac{\delta_{A} + \delta_{B}}{2}\Bigr)$. Combining Theorem \ref{thm:2D_FUP} with Proposition 6.8 from \cite{DyatlovFUPSurvey}, we can classify exactly which discrete 2D Cantor sets exhibit a fractal uncertainty principle. 

\begin{corollary}\label{cor:sharp_FUP_full_range}
Let $\mc A, \mc B$ be a pair of alphabets. Equation (\ref{eq:2D_FUP_formula}) holds for some $\beta > \max\Bigl(0, 1 - \frac{\delta_A + \delta_B}{2}\Bigr)$ if and only if 
\begin{itemize}
	\item $\delta_A + \delta_B \geq 2$ and the orthogonal line condition from Theorem \ref{thm:2D_FUP} holds, 
	\item $\delta_A + \delta_B \leq 2$ and for some $\mb j, \mb j' \in \mc A$, $\mb k, \mb  k' \in \mc B$, 
	\begin{equation*}
		\langle \mb j - \mb j', \mb k - \mb k'\rangle \neq 0 \text{ as an inner product in } \Z.
	\end{equation*}
\end{itemize}
\end{corollary}
The second condition above is a different sort of orthogonal line condition from the first. Although it is not initially obvious, the two conditions are the same when $\delta_A + \delta_B = 2$. Indeed, this must be the case, because both conditions are if and only if statements. If $\delta_A + \delta_B = 2$ and $\mc A, \mc B$ do not obey a FUP, then $\delta_A = \delta_B = 1$ and 
\begin{align*}
	\mc A &= \{(x_0,t)\, :\, t \in \Z_M\} \text{ and } \mc B = \{(t,y_0)\, :\, t \in \Z_M\} \text{ for some } x_0, y_0 \in \Z_M, \text{ or } \\ 
	\mc A &= \{(t,t)\, :\, t \in \Z_M\} \text{ and } \mc B = \{(t, M-1-t)\, :\, t \in \Z_M\}, 
\end{align*}
or the reverse of these. Indeed, if $\delta_A < 1$ then $\mb X$ is less than one dimensional and it cannot contain any line, so Theorem \ref{thm:2D_FUP} applies. If $\delta_A = \delta_B = 1$ then $|\mc A| = |\mc B| = N$, and $\mc A - \mc A$, $\mc B - \mc B$ must both lie on one dimensional cosets as subsets of $\Z^2$. This can only be true in one of the two cases listed above. 

\subsection{Sketch of the argument}
Suppose $f: \Z_N^2 \to \C$ has $\supp f = S$, $\supp \hat f = T$. Our argument shows that if $S$ avoids lines in a robust sense, then $|T| \gtrsim N^2$. Proposition \ref{prop:supp_size_cor} is a realization of this heuristic. 

We start by writing functions on $\Z_N^2$ with Fourier support in $[0,D]^2$ as a trigonometric polynomial on $\T^2 \subset \C^2$ with degree $\leq D$. We gain two things from using polynomials: unique factorization and Bezout's theorem on the intersection of zero loci. The heart of the argument is constructing a trigonometric polynomial
\begin{equation}\label{eq:trig_poly_g}
	h(x,y) = \sum_{0 \leq k, l \leq D} a_{kl} z^k w^l,\quad z = e^{\frac{2\pi i}{N} x}, w = e^{\frac{2\pi i}{N} y},\quad D \lesssim \sqrt{|T|}
\end{equation}
which vanishes on all of $T$ except one line (and does not vanish on all of $T$). Then $h \hat f$ is nonzero and supported along a line, so $(h \hat f)^{\vee}$ has constant magnitude along dual lines. We have $(h \hat f)^{\vee} = h^{\vee} * f$, so 
\begin{equation*}
	\supp (h \hat f)^{\vee} \subset S - [0,D] \times [0, D].
\end{equation*}
Thus $S - [0,D] \times [0,D]$ contains \textit{some} dual line, and combining this fact with the structural condition on $S$ implies $D \gtrsim N$. Thus $|T| \gtrsim N^2$. Because we end up analyzing the function $h \hat f$, $h$ is called a \textit{multiplier}. 

It is useful to consider a hypothetical scenario: what if $T$ is the vanishing set of some low degree trigonometric polynomial in $\Z_N^2$, e.g. 
\begin{equation*}
	T = \{(x,y) \in \Z_N^2\, :\, z^2 + 4zw + w = 1\},\quad  z = e^{\frac{2\pi i}{N} x}, w = e^{\frac{2\pi i}{N} y}. 
\end{equation*}
Bezout's inequality (Theorem \ref{thm:bezout_theorem}) states that any trigonometric polynomial $h$ can only vanish on at most $4D$ points of $T$, or it must vanish on all of $T$. So any multiplier as in (\ref{eq:trig_poly_g}) would have degree $\sim |T| \gg \sqrt{|T|}$, obstructing our strategy if $|T|$ is large.

Luckily, Theorem \ref{thm:beukers_smyth} of Ruppert \cite{Ruppert}*{Corollary 5} and Beukers \& Smyth \cite{BeukersSmyth}*{Theorem 4.1} excludes this possibility. They prove that the vanishing set of a degree $D$ trigonometric polynomial in $\Z_N^2$ either has order $\leq 22D^2$, or contains a line. Concretely, with $T$ defined as above, $|T| \leq 88$ for all $N$. This theorem gives a sharp quantitative form to Lang's conjecture, which is a qualitative statement about cyclotomic roots of polynomials in $\C^n$---see \S\ref{sec:main_lemma} for more details. Lemma \ref{lem:main_lemma} encapsulates this number theoretic input as it applies to our result. 

\subsection{An application to quantum chaos}\label{sec:application_quantum_chaos}
Dyatlov and Jin \cite{DyatlovJinBakersMaps} initially introduced Theorem \ref{thm:1D_FUP} to prove results in quantum chaos. In particular, they used Theorem \ref{thm:1D_FUP} to prove a class of one dimensional \textit{quantum open baker's maps}, a discrete model for open quantum maps, always have a spectral gap. Adapting their pipeline we can use our Theorem \ref{thm:2D_FUP} to prove a large class of two dimensional quantum open baker's maps have a spectral gap. 

\subsubsection*{One dimensional baker's maps}
First we will review the one dimensional situation as discussed in \cite{DyatlovJinBakersMaps}. The quantum open baker's maps in consideration are parametrized by triples 
\begin{equation*}
	(M, \mc A, \chi),\quad M \in \Z_{> 0},\quad \mc A \subsetneq \Z_M,\quad \chi \in C_0^{\infty}((0,1); [0,1]). 
\end{equation*}
Here $M$ is the base, $\mc A$ is the alphabet, and $\chi$ is the cutoff function. For any $N \geq 1$, let $\chi_N \in \ell^2(\Z_N)$ be given by $\chi_N(x) = \chi(x/N)$. For each $k \geq 1$ the corresponding quantum open baker's map is the operator on $\ell^2(\Z_N)$, $N=M^k$, given by 
\begin{align*}
	B_N &= \mc F_N^* \begin{pmatrix}
		\chi_{N/M} \mc F_{N/M} \chi_{N/M} & & \\
		& \ddots & \\
		& & \chi_{N/M} \mc F_{N/M} \chi_{N/M}
	\end{pmatrix}
	I_{\mc A, M}. \\ 
\end{align*}
Here $I_{\mc A, M}$ is the $N\times N$ diagonal matrix with $k$-th diagonal entry equal to 1 if $\lfloor \frac{\ell}{N/M}\rfloor \in \mc A$ and 0 otherwise, and $\chi_{N/M} \mc F_{N/M} \chi_{N/M}$ is an $(N/M)\times (N/M)$ block matrix given by the corresponding operator on $\ell^2(\Z_{N/M})$. It is convenient to introduce the projection operator 
\begin{align*}
	&\Pi_a: \ell^2(\Z_N) \to \ell^2(\Z_{N/M}),\quad a \in \Z_M,\\
	&\Pi_a u(j) = u\Bigl(j + a\frac{N}{M}\Bigr).
\end{align*} 
Then 
\begin{align*}
	B_N = \sum_{a \in \mc A} B_N^a,\quad B_N^a := \mc F_N^* \Pi_a^* \chi_{N/M} \mc F_{N/M} \chi_{N/M} \Pi_a. 
\end{align*}
Let $\mc X_k \subset \Z_{M^k}$ denote the Cantor iterates of $\mc A$ as before. 
The following proposition relates the fractal uncertainty principle to spectral gaps for $B_N$. 
\begin{proposition}[\cite{DyatlovJinBakersMaps}*{Proposition 2.6}]\label{prop:dyatlov_jin_FUP_specgap}
Suppose 
\begin{equation}\label{eq:FUP_eqn_for_FUP_implies_specgap}
	\|1_{\mc X_k} \mc F 1_{\mc X_k}\|_{2\to 2} \leq C_{\beta}M^{-k \beta}\quad \text{ for all } k.
\end{equation}
Then 
\begin{equation}\label{eq:spec_gap}
	\limsup_{N\to \infty} \max\{|\lambda|\, :\, \lambda \in \Sp(B_N)\} \leq M^{-\beta},
\end{equation}
where $\Sp(B_N)$ is the spectrum.
\end{proposition}

Combining Proposition \ref{prop:dyatlov_jin_FUP_specgap} with Theorem \ref{thm:1D_FUP}, Dyatlov and Jin obtain a spectral gap for our quantum open bakers maps. 
\begin{theorem}[\cite{DyatlovJinBakersMaps}*{Theorem 1}]\label{thm:1D_spec_gap_theorem}
There exists $\beta = \beta(M, \mc A) > 0$ such that 
\begin{equation*}
	\limsup_{N\to \infty} \max\{|\lambda|\, :\, \lambda \in \Sp(B_N)\} \leq M^{-\beta},
\end{equation*}
$\Sp(B_N)$ is the spectrum.
\end{theorem}
It is not hard to show that (\ref{eq:FUP_eqn_for_FUP_implies_specgap}) always holds with $\beta = \max(0, 1/2 - \delta)$, $\delta$ the fractal dimension, so this theorem is only interesting when $\delta \geq 1/2$. A different argument for $\delta < 1/2$ shows that in Theorem \ref{thm:1D_FUP} we can take $\beta > \max(0, 1/2 - \delta)$ for all $\delta$, giving an improved spectral gap for all fractal dimensions in Theorem \ref{thm:1D_spec_gap_theorem}. 

\subsubsection*{Two dimensional baker's maps}
A two dimensional quantum open baker's map is parametrized by a triple 
\begin{equation*}
	(M, \mc A, \chi),\quad M \in \Z_{> 0},\quad \mc A \subsetneq (\Z_M)^2,\quad \chi \in C_0^{\infty}((0,1)^2; [0,1]).
\end{equation*}
We will define baker's maps $B_N: \ell^2(\Z_N^2) \to \ell^2(\Z_N^2$), $N = M^k$. 
As before, define 
\begin{align*}
	&\Pi_{\mb a}: \ell^2(\Z_N^2) \to \ell^2(\Z_{N/M}^2),\quad \mb a = (a_1, a_2) \in (\Z_M)^2, \\ 
	&\Pi_{\mb a} u(\mb j) = u\Bigl(\mb j + \mb a\frac{N}{M}\Bigr).
\end{align*}
Then set 
\begin{align*}
	B_N = \sum_{\mb a \in \mc A} B_N^{\mb a}, \quad B_N^{\mb a} := \mc F_N^* \Pi_{\mb a}^* \chi_{N/M} \mc F_{N/M} \chi_{N/M} \Pi_{\mb a}
\end{align*}
where $\mc F_N$ denotes the unitary Fourier transform on $\ell^2(\Z_N^2)$ and $\chi_N(\mb j) = \chi(\mb j/N)$. In \S\ref{sec:proof_application_bakers} we sketch the proof that Proposition \ref{prop:dyatlov_jin_FUP_specgap} holds for two dimensional bakers maps as well, leading to the following. 

\begin{theorem}\label{thm:2D_spec_gap_theorem}
Suppose $\mc A \subsetneq \Z_M^2$ is an alphabet such that $\mb X$, the Cantor set generated by $\mc A$, does not contain a pair of orthogonal lines as in Theorem \ref{thm:2D_FUP}. Then there is some $\beta  = \beta(M, \mc A) > 0$ so that 
\begin{align*}
	\limsup_{N\to \infty} \max\{|\lambda|\, :\, \lambda \in \Sp(B_N)\} \leq M^{-\beta}.
\end{align*}
\end{theorem}
Just as Theorem \ref{thm:1D_spec_gap_theorem} is only interesting for $\delta \geq 1/2$, Theorem \ref{thm:2D_spec_gap_theorem} is only interesting for $\delta \geq 1$, because we can always take $\beta = \max(0, 1 - \delta)$ in (\ref{eq:FUP_eqn_for_FUP_implies_specgap}). 

\subsection{Organization}
In \S\ref{sec:1D_argument} we give a new proof of 1D FUP (Theorem \ref{thm:1D_FUP}) as a warmup for our 2D argument. In \S\ref{sec:2D_argument} we prove Theorem \ref{thm:2D_FUP}, up to the proof of the main Lemma \ref{lem:main_lemma}, which we defer to \S\ref{sec:main_lemma}. 
In \S\ref{sec:loose_ends} we supply proofs of several earlier claims which are not directly relevant to Theorem \ref{thm:2D_FUP}.
In particular, we show the condition of Theorem \ref{thm:2D_FUP} is sharp,
prove Proposition \ref{prop:lines_in_cantor_set} regarding lines in Cantor sets, and sketch the 2D proof of Proposition \ref{prop:dyatlov_jin_FUP_specgap} regarding the application of FUP to quantum Bakers maps. 
In \S\ref{sec:beukers-smyth-argument} we give a sketch of Ruppert and Beukers-Smyth's Theorem \ref{thm:beukers_smyth} which is the essential ingredient to our Lemma \ref{lem:main_lemma}. 
Finally, in \S\ref{sec:higher_dim_continuous}, we compare Theorem \ref{thm:2D_FUP} to a more recent higher dimensional FUP the author \cite{Cohen2023} proved in $\R^d$. The more recent result can be used to prove an FUP for discrete Cantor sets in any dimension that avoid all lines, but cannot recover the precise orthogonal line condition proved in 2D in the present paper. 

\section{The 1D argument}\label{sec:1D_argument}
Our starting point is the following simple argument which can be used to establish a 1D FUP. 

\begin{proposition}\label{prop:1D_prop}
Let $I = [a, b)$ be an interval, and suppose $f: \Z_N \to \C$ is nonzero and has $\hat f|_{I} = 0$. Then $|\supp f| > |I| = b-a$.
\end{proposition}
\begin{proof}
Suppose $|\supp f| = k$. Let $S = \supp f = \{x_1, \ldots, x_k\}$. Let $F(z)$ be the polynomial 
\begin{align*}
	F(z) = (z-e^{\frac{2\pi i}{N} x_1})\dots (z-e^{\frac{2\pi i}{N} x_{k-1}}) = \sum_{j=0}^{k-1} a_j z^j. 
\end{align*}
Let $h: \Z_N\to \C$ be defined by
\begin{align*}
	h(x) = \frac{1}{\sqrt{N}} F(e^{\frac{2\pi i}{N} x}),\quad \hat h(j) = \begin{cases}
		a_j & 0 \leq j \leq k-1, \\ 
		0 & \text{else}.
	\end{cases}
\end{align*}
Then $h$ vanishes on all of $S$ except for $x_k$ (and $h$ is nonzero at $x_k$). Thus $hf = c \delta_{x_k}$, $c \neq 0$. So $\widehat{hf}$ has full Fourier support. But 
\begin{align*}
	\widehat{hf}(b-1) = (\hat h * \hat f)(b-1) = \sum_{j=0}^{k-1}  \hat h(j) \hat f(b-1-j).
\end{align*}
If $k \leq |I|$ we have $\widehat{hf}(b-1) = 0$ leading to a contradiction. Thus $|\supp f| > |I|$. 
\end{proof}
See Figure \ref{fig:1d_arg_vis} for a visualization. 
\begin{figure}
\includegraphics[width=0.9\linewidth]{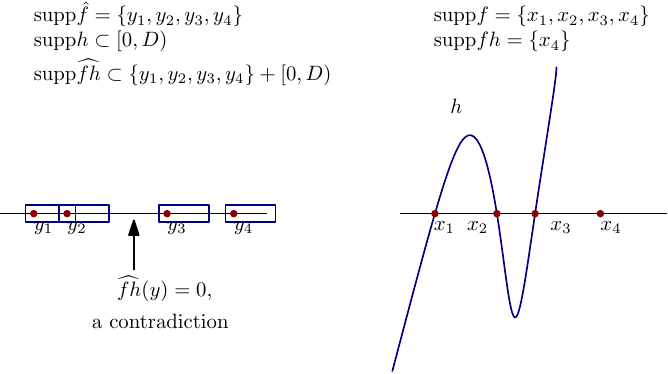}
\caption{Diagram of the 1D argument}
\label{fig:1d_arg_vis}
\end{figure}

\begin{remark}
This proof shares some similarities with Bourgain-Dyatlov's proof of 1D FUP for general fractal sets. They constructed a function $\psi$ with compact Fourier support and which decays quickly on a fractal set. They multiply by this function to discover that a function supported on a fractal set must have substantial Fourier mass in a union of intervals. In the discrete setting, things are much simpler: we may construct a multipler that vanishes on all but one element of the fractal set, and then multiply by this function to discover some Fourier mass in every gap. 
\end{remark}

\section{The 2D argument}\label{sec:2D_argument}

We first state our main lemma, then derive Theorem \ref{thm:2D_FUP} from this lemma, and finally discuss the proof of the lemma. 
For $A \subset \Z_N^2$, let 
\begin{align*}
	\N_R(A) &= A + [0,R) \times [0,R) = \supp (1_{[0,R)\times [0,R)} * 1_A)
\end{align*}
be the $R$-neighborhood of $A$. A line $\ell \subset \Z_N^2$ is a coset of the form 
\begin{align*}
	\ell = \{(x,y) \in \Z_N^2\, :\, ax+by=c\}.
\end{align*}
The coefficients $(a,b,c)$ are only determined up to multiplication by $\Z_N^{\times}$. 
We say $\ell$ is \textit{irreducible} if $a,b$ are coprime over $\Z_N$, and $\|\ell\| = R$ is the minimal number so that we can write
\begin{equation}\label{eq:size_of_line}
	\ell = \{(x,y) \in \Z_N^2\, :\, ax+by = c\},\quad |a|,|b| \leq R. 
\end{equation}

\begin{lemma}\label{lem:main_lemma}
Let $f: \Z_N^2 \to \C$ be a nonzero function with $\supp f = S$. Let $R = \lfloor 200|S|^{1/2}\rfloor$. There is an irreducible line $\ell$ with $\|\ell\| \leq R$ and a nonzero function $g$ with $\supp g \subset S \cap \ell$ and $\supp \hat g \subset \N_{R}(\supp \hat f)$. 
\end{lemma}
This Lemma is analagous to the proof of Proposition \ref{prop:1D_prop}, except we can only localize the support of $f$ to a line $\ell$ rather than to a single point. See Figure \ref{fig:lemma1_vis}. 
\begin{figure}
\includegraphics[width=0.8\linewidth]{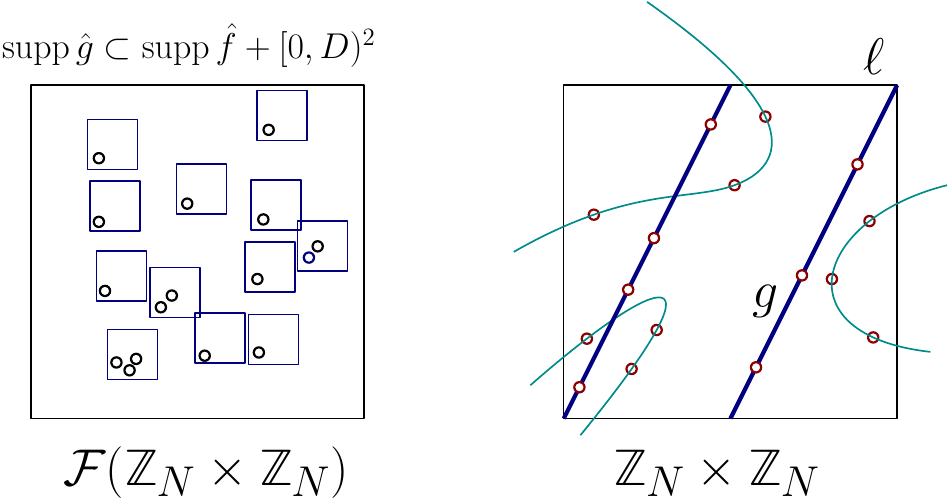}
\caption{Visualization of Lemma \ref{lem:main_lemma}}
\label{fig:lemma1_vis}
\end{figure}
Before showing how to derive Theorem \ref{thm:2D_FUP} using this Lemma we discuss discretizations of sets in $\T^2$, lines in $\T^2$, and lines in $\Z_N^2$. 

\subsection{Discretization of fractal sets}\label{sec:disc_of_fractal_set}
It will be more convenient to state our main results for discretizations of general fractal sets in $\T^2$ and then specialize to Cantor sets later. Let $\mb X \subset \T^2$ be closed. For $0 < r < 1$, let $\N_r(\mb X) = \mb X + [-r,r] \times [-r,r]$ be the $r$-neighborhood.\footnote{Our convention is that $\N_r(\mb X) = \mb X + [-r,r]\times [-r,r]$ denotes a neighborhood in $\T^2$, and $\N_R(A) = A + [0,R)\times [0,R)$ denotes an ``upper right'' neighborhood in $\Z_N^2$. We take the full neighborhood in $\T^2$ rather than just the upper right neighborhood for technical reasons---this convention makes Equation (\ref{eq:nbhd_eq}) true, and otherwise it would be more complicated to state.} Let
\begin{align*}
	X_N &= \left\{(x,y) \in \Z_N^2\, :\, \Bigl[\frac{x}{N}, \frac{x+1}{N}\Bigr] \times \Bigl[\frac{y}{N}, \frac{y+1}{N}\Bigr] \cap \mb X \neq \emptyset\right\} \\ 
	&\subset \left\{(x,y) \in \Z_N^2\, :\, \Bigl(\frac{x}{N}, \frac{y}{N}\Bigr) \in \N_{1/N}(\mb X)\right\}
\end{align*}
be a discretization of $\mb X$ to $\Z_N^2$. 
If $\mb X$ is the limiting Cantor set for an alphabet $\mc A$, then $X_{M^k} \subset \Z_{M^k}^2$ is just slightly larger than the $k$th Cantor iterate $\mc X_k$ of $\mc A$ (due to endpoint considerations). Likewise, the drawing $\mb X_k$ (\ref{eq:kth_iterate_drawing}) of the $k$th iterate in $\T^2$ is slightly smaller than $\N_{M^{-k}}(\mb X)$. 
If $R$ is an integer and $\N_R(X_N) = X_N + [0,R)\times[0,R)$, then 
\begin{equation}\label{eq:nbhd_eq}
	\N_R(X_N) \subset \left\{(x, y) \in \Z_N^2\, :\, \Bigl(\frac{x}{N}, \frac{y}{N}\Bigr) \in \N_{R/N}(\mb X)\right\} 
\end{equation}
where $\N_{R/N}(\mb X) \subset \T^2$. In what follows $R$ will be $\sim N^{\beta}$, $\beta < 1$, so $R/N \sim N^{\beta - 1}$, and $\N_{R/N}(\mb X)$ will look like a very small neighborhood of $\mb X$ in $\T$. 

\subsection{Some useful lemmas on lines}

\begin{lemma}\label{lem:lines_gen_one_elem}
Let 
\begin{align*}
	\ell = \{(x,y) \in \Z_N^2\, :\, ax+by=c\},\quad a,b,c \in \Z_N
\end{align*}
be an irreducible line, i.e., $a,b$ are coprime as elements of $\Z_N$. Then $\ell = \Z(-b, a) + \mb p$ where $\mb p \in \ell$ is arbitrary. We have $|\ell| = N$. Also, $a, b$ can be taken as coprime integers. 
\end{lemma}
\begin{proof}
Pick $s, t$ so tha $sa+tb = 1 \pmod N$. We have $(cs, ct) \in \ell$. Suppose $ax+by=0$. We claim $(x,y) = (-b, a)\cdot (-tx+sy)$. Indeed, 
\begin{align*}
	-b(-tx+sy) = tbx - sby = tbx + sax = x \pmod N, \\ 
	a(-tx+sy) = -atx + asy = tby + say = y \pmod N
\end{align*}
as needed. This shows that for $(x,y) \in \ell$, $(x, y) - \mb p \in (-b, a)\Z$. 

To see that $|\ell| = N$, notice that $(-nb, na) + \mb p = (-mb, ma) + \mb p \pmod N$ if and only if $(-(n-m)b, (n-m)a) = 0 \pmod N$ if and only if $n = m \pmod N$, using that $a,b$ are coprime. 

Finally, suppose $a$ and $b$ are not coprime integers, but $a = \alpha a', b = \alpha b'$, where $a', b'$ are coprime integers. Then because $a,b$ are coprime mod $N$, $\alpha, N$ are coprime, so 
\begin{align*}
	ax+by = c \Leftrightarrow \alpha(a'x+b'y) = c \Leftrightarrow a'x + b'y = \alpha^{-1}c
\end{align*}
where the equalities above are mod $N$. 
\end{proof}

We will need a uniformity result for lines through closed sets $\mb X \subset \T^2$. In what follows 
\begin{equation}\label{eq:linf_dist_T2}
d(\mb p, \mb q) = \max(|p_1 - q_1|_{\T}, |p_2 - q_2|_{\T}),\quad |x|_{\T} = \min_{n \in \Z} |x-n|_{\R}. 	
\end{equation}
is the $\ell^{\infty}$ distance on $\T^2$. 
First we need a lemma. 
\begin{lemma}\label{lem:denseness_fact}
Let $\mb v = (a,b)$ with $a,b$ coprime integers. Every coset $\ell = \R \mb v + \mb p$ is quantitatively dense in $\T^2$, in the sense that for every $\mb q \in \T^2$, $d(\mb q, \ell) \leq 1/\max(|a|, |b|)$.
\end{lemma}
In the following proof we denote $\frac{1}{b}\Z = \{n / b\, :\, n \in \Z\}$. 
\begin{proof}
For every $y_0 \in \T$, $(\R \mb v + \mb p) \cap \{y = y_0\}$ is a coset of $\frac{1}{b}\Z$, 
and for every $x_0 \in \T$, $(\R \mb v + \mb p) \cap \{x = x_0\}$ is a coset of $\frac{1}{a}\Z$. Thus
\begin{align*}
	d((x_0,y_0), \ell) &\leq d((x_0,y_0), \ell \cap \{y = y_0\}) \leq \frac{1}{|a|} \\ 
	d((x_0,y_0), \ell) &\leq d((x_0,y_0), \ell \cap \{x = x_0\}) \leq \frac{1}{|b|} 
\end{align*}
giving the result. 
\end{proof}

\begin{lemma}\label{lem:line_contains_uniformity}
Suppose $\mb X \subsetneq \T^2$ is closed. There is a constant $c_{\mb X} > 0$ such that for every direction $\mb v \in \R^2 - \{0\}$, either some coset $\R \mb v + \mb p$ lies entirely on $\mb X$, or 
\begin{equation}\label{eq:line_avoids_byc}
 	\sup_{x \in \R \mb v + \mb p} d(x, \mb X) \geq c_{\mb X}
\end{equation} 
for every $\mb p$. Moreover, there is some $C_{\mb X} > 0$ so that if $a, b$ are coprime integers with $\max(|a|, |b|) > C_{\mb X}$, then (\ref{eq:line_avoids_byc}) holds for $\mb v = (a,b)$. 
\end{lemma}
\begin{proof}
Because $\mb X$ is a closed proper subset of $\T^2$ it is not dense, and there is some $x_0 \in \T^2$ with $d(x_0, \mb X) \geq 2c_0$. If $\mb v = (\alpha, \beta)$ with $\alpha / \beta$ or $\beta / \alpha$ irrational, then $\R \mb v + \mb p$ is dense and has points coming arbitrarily close to $x_0$, thus 
\begin{align*}
	\sup_{x \in \R \mb v + \mb p} d(x, \mb X) \geq 2c_0.
\end{align*}
Otherwise, let $\mb v = (a,b)$ with $a,b$ coprime integers. By Lemma \ref{lem:denseness_fact},
\begin{align*}
	\inf_{x \in \R \mb v + \mb p} d(x, x_0) \leq 1/\max(|a|, |b|)
\end{align*}
and
\begin{align*}
	\sup_{x \in \R \mb v + \mb p} d(x, \mb X) \geq 2c_0 - 1/\max(|a|, |b|). 
\end{align*}
Hence if $\max(|a|, |b|) > 1/c_0$, $\sup_{x \in \R \mb v + \mb p} d(x, \mb X) \geq c_0$. For each pair of coprime integers $a,b$ with $\max(|a|, |b|) \leq 1/c_0$, either some coset $\R(a,b) + \mb p$ lies on $\mb X$, or there is a $c_1$ so 
\begin{align*}
 	\sup_{x \in \R(a,b) = \mb p} d(x, \mb X) \geq c_1 \quad \text{ for all } \mb p \in \T^2.
\end{align*} 
There are finitely many such choices of $(a,b)$, so $c_1$ can be chosen uniformly in all of them. We take $c_{\mb X} = \min(c_0, c_1)$ in (\ref{eq:line_avoids_byc}). 
\end{proof}

\subsection{Proof of Theorem \ref{thm:2D_FUP} assuming Lemma \ref{lem:main_lemma}}
Before proving Theorem \ref{thm:2D_FUP}, we prove the following simpler proposition, which applies when one of the fractal sets $\mb X, \mb Y$ avoids all lines. 

\begin{proposition}\label{prop:supp_size_cor}
Suppose $\mb X \subsetneq \T^2$ is closed and does not contain any closed cosets $\R \mb v + \mb p \subset \T^2$. By Lemma \ref{lem:line_contains_uniformity}, there is some $c_{\mb X} > 0$ so that
\begin{equation*}
	\sup_{x \in \ell} d(x, \mb X) \geq c_{\mb X}, \quad \ell = \R \mb v + \mb p \text{ arbitrary}.
\end{equation*}
If $f: \Z_N^2 \to \C$ is nonzero and has $\supp \hat f \subset X_N$, then 
\begin{equation}\label{eq:S_large}
	|\supp f| \geq \frac{c_{\mb X}^2}{400^2}N^2.
\end{equation}
\end{proposition}
\begin{proof}
Suppose $\supp f = S$, $\supp \hat f \subset X_N$. 
Apply Lemma \ref{lem:main_lemma} to $f$. We obtain an $R \leq 200 |S|^{1/2}$, a line 
\begin{align*}
	\ell = \{(x, y)\, :\, ax+by=c\},\quad a,b \text{ coprime},\quad \max(|a|, |b|) \leq R
\end{align*}
and a nonzero $g$ supported on $\ell$ with $\supp \hat g \subset \N_{R}(X_N)$. We claim $R/N \geq c_{\mb X}/2$, which would imply (\ref{eq:S_large}). 

Suppose $R/N < c_{\mb X}/2$. We show $g = 0$. Set $\mb v = (a,b)$ and $\mb v^{\perp} = (-b, a)$.
Because $g$ is supported on $\ell$, $\hat g$ has constant magnitude on dual lines $\Z\mb v + \mb p$. Indeed,
\begin{align*}
	\hat g(\mb \xi) &= \frac{1}{N}\sum_{\mb v \cdot \mb x = c \pmod N} g(\mb x) e^{\frac{2\pi i}{N} \mb \xi \cdot \mb x} \\ 
	\hat g(\mb \xi + n\mb v) &= \frac{1}{N}\sum_{\mb v \cdot \mb x = c \pmod N} g(\mb x) e^{\frac{2\pi i}{N} n\mb v \cdot \mb x}e^{\frac{2\pi i}{N} \mb \xi \cdot \mb x}  = e^{\frac{2\pi i}{N} nc} \hat g(\mb \xi).
\end{align*}
Let $\mb \xi \in \Z_N^2$ be arbitrary. Let $t \in \R$ be such that $d(t\mb v/N + \mb \xi/N, \mb X) \geq c_{\mb X}$.
Let $n$ be the nearest integer to $t$. Then 
\begin{align*}
	d(n\mb v/N + \mb \xi/N, \mb X) \geq c_{\mb X} - \max\Bigl(\frac{|a|}{N}, \frac{|b|}{N}\Bigr) \geq c_{\mb X}/2.
\end{align*}
By Equation (\ref{eq:nbhd_eq}), since $R/N < c_{\mb X}/2$, $n\mb v + \mb \xi \notin \N_R(X_N)$, so $\hat g(n\mb v + \mb \xi) = 0$ by hypothesis. Thus $\hat g(\mb \xi) = 0$ as well. Since $\mb \xi \in \Z_N^2$ was arbitrary, $g = 0$.
\end{proof}

Now we prove a more general proposition applying to measure zero sets $\mb X, \mb Y$ which don't contain a pair of orthogonal lines. Theorem \ref{thm:2D_FUP} follows directly from this proposition by submultiplicativity. 

\begin{proposition}\label{prop:main_FUP_prop}
Suppose $\mb X, \mb Y \subset \T^2$ are closed and have Lebesgue measure zero. Suppose that for every direction $\mb v = (a,b) \in \R^2 - \{0\}$, $\mb v^{\perp} = (-b, a)$ either $\mb X$ contains no coset $\R \mb v + \mb p$ or $\mb Y$ contains no coset $\R \mb v^{\perp} + \mb p$. Then for large enough $N$, there is no nonzero $f: \Z_N^2 \to \C$ with $\supp f \subset X_N$ and $\supp \hat f \subset Y_N$. 	
\end{proposition} 
The proof involves two cases, see Figure \ref{fig:cases_in_main_FUP_prop}.
\begin{figure}
\begin{subfigure}{0.8\textwidth}
\centering
\includegraphics[width=\linewidth]{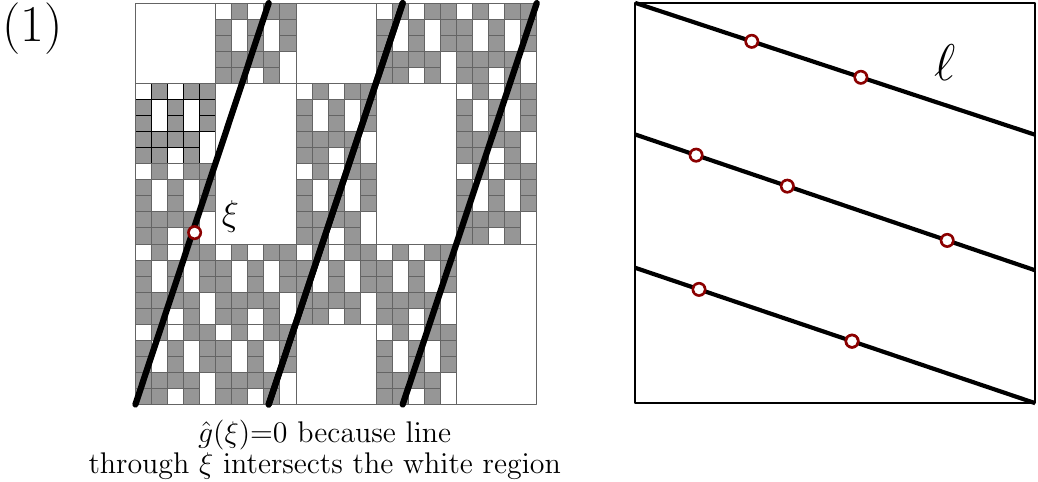}
\end{subfigure}
\\\vspace{0.2in}
\begin{subfigure}{0.8\textwidth}
\centering
\includegraphics[width=\linewidth]{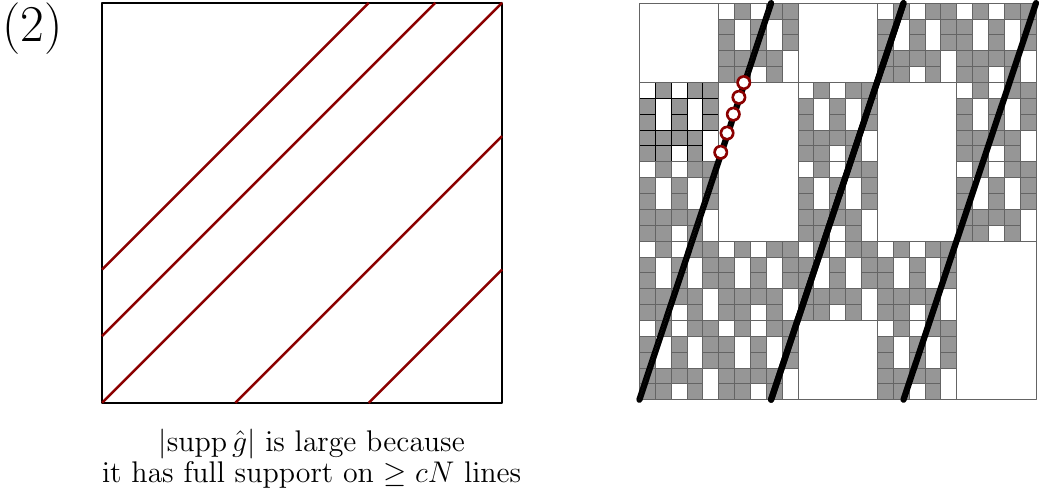}
\end{subfigure}
\caption{The two cases in Proposition \ref{prop:main_FUP_prop} obtain contradictions in different ways.}
\label{fig:cases_in_main_FUP_prop}
\end{figure}

\begin{proof}
First notice that by continuity of measure, 
\begin{equation}\label{eq:measure_to_zero}
	\lim_{r\to 0} |\N_r(\mb X)| = \lim_{r\to 0} |\N_r(\mb Y)| = 0
\end{equation}
where $|\cdot|$ denotes the Lebesgue measure. It follows that
\begin{equation}
	|X_N|, |Y_N| = o(N^2)
\end{equation}
as $N\to \infty$. 

Using the hypothesis and Lemma \ref{lem:line_contains_uniformity}, there is some $c > 0$ such that for every coprime $a,b$, either 
\begin{align}
	\sup_{y \in \R(a,b) + \mb p} d(y, \mb Y) &\geq c \text{ for all $\mb p$, or } \label{eq:y_dist_large} \\
	\sup_{x \in \R(-b,a) + \mb p} d(x, \mb X) &\geq c \text{ for all $\mb p$.} \label{eq:x_dist_large}
\end{align}
There is also some $C > 0$ so that if $\max(|a|, |b|) > C$, then (\ref{eq:y_dist_large}) and (\ref{eq:x_dist_large}) both hold. 

Suppose $\supp f = S \subset X_N$ and $\supp \hat f \subset Y_N$. Apply Lemma \ref{lem:main_lemma} to $f$ to obtain an $R \leq o(N)$, a line 
\begin{align*}
	\ell = \{(x, y)\, :\, ax+by=c\},\quad a,b \text{ coprime},\quad \max(|a|, |b|) \leq R
\end{align*}
and a nonzero $g$ supported on $\ell \cap X_N$ with $\supp \hat g \subset \N_R(Y_N)$. Let $\mb v = (a,b)$, $\mb v^{\perp} = (-b, a)$. 

\par{\textbf{Case 1.}} Suppose (\ref{eq:y_dist_large}) holds. Then we are in the same position as Proposition \ref{prop:supp_size_cor}, and for $N$ large enough we conclude $g = 0$ which is a contradiction.

\par{\textbf{Case 2.}} Suppose (\ref{eq:y_dist_large}) does not hold. Then (\ref{eq:x_dist_large}) holds and $\max(|a|, |b|) \leq C$. 
Choose $\mb p = (p_1, p_2) \in \ell$, so $\ell = \Z\mb v^{\perp} + \mb p$. Write $g(n\mb v^{\perp} + \mb p) = \tilde g(n)$. Then
\begin{align*}
	\hat g(\mb \xi) &= \frac{1}{N}\sum_{n \in \Z_N} \tilde g(n)e^{-\frac{2\pi i}{N} \mb \xi \cdot (n\mb v^{\perp} + \mb p)} \\ 
	&= e^{-\frac{2\pi i}{N} \mb \xi \cdot \mb p} N^{-1}\sum_{n\in \Z_N} \tilde g(n) e^{-\frac{2\pi i}{N} n \mb \xi \cdot \mb v^{\perp}}
\end{align*}
Notice in particular that $\hat g$ only depends on $\mb \xi \cdot \mb v^{\perp}$. By Lemma \ref{lem:lines_gen_one_elem}, for every $d \in \Z_N$ there are $N$ solutions in $\mb \xi$ to $\mb \xi \cdot \mb v^{\perp} = d$. So we may write 
\begin{align*}
	\hat g(\mb \xi) &= \frac{1}{\sqrt{N}}e^{-\frac{2\pi i}{N} \mb \xi \cdot \mb p} \hat {\tilde g}(\mb \xi \cdot \mb v^{\perp}) = \frac{1}{N}\sum_{n\in \Z_N} \tilde g(n) e^{-\frac{2\pi i}{N} n \mb \xi \cdot \mb v^{\perp}} \\ 
	|\hat g(\mb \xi)| &= \frac{1}{\sqrt{N}}|\hat{\tilde g}(\mb \xi \cdot \mb v^{\perp})|
\end{align*}
Thus $|\supp \hat g| = N |\supp \hat {\tilde g}|$. 

Choose $t \in \R$ so that $d(t\mb v^{\perp}/N + \mb p/N, \mb X) \geq c.$. 
Then
\begin{align*}
	d(s\mb v^{\perp}/N + \mb p/N, \mb X) &\geq c - |s-t|\frac{C}{N} \\ 
	&\geq c/2 \quad \text{ for } |s-t| \leq \frac{c}{2C}N.
\end{align*}
If $s$ is an integer satisfying the above and $N > \frac{100}{c}$, we conclude that $s\mb v^{\perp} + \mb p \notin X_N$. 

 Let $I = [t - (c/2C)N, t+(c/2C)N] \cap \Z$. Then $|I| \geq \frac{c}{C} N$ and $\tilde g|_I = 0$. By Proposition \ref{prop:1D_prop}, 
\begin{align*}
	|\supp \hat {\tilde g}| = N|\hat {\tilde g}| \geq \frac{c}{C} N^2.
\end{align*}
On the other hand, $|Y_N| \leq o(N^2)$, leading to a contradiction for large enough $N$. 
\end{proof}


\begin{remark}
Although Proposition \ref{prop:main_FUP_prop} applies to arbitrary fractal sets, Theorem \ref{thm:2D_FUP} only applies to Cantor sets, because we need submultiplicativity in order to prove exponential decay (\ref{eq:2D_FUP_formula}). 
\end{remark}

\begin{remark}
Let us note a quantitative difference between cases 1 and 2 above. In case 1 the coefficients $a,b$ determining the line $\ell$ should have size $o(N)$ in order to obtain a contradiction. In case 2, the cofficients $a,b$ must have size $O(1)$. Lemma \ref{lem:main_lemma} only gives $a,b = o(N)$, so we argue that if $\max(|a|,|b|) > C$ we can use Case 1, and Case 2 only arises when $\max(|a|,|b|) \leq C$. 
\end{remark}

Now we can conclude Theorem \ref{thm:2D_FUP}.

\begin{proof}[Proof of Theorem \ref{thm:2D_FUP}]
Suppose $\mc A$ and $\mc B$ are alphabets satisfying the condition of Theorem \ref{thm:2D_FUP}. The Cantor sets they generate, $\mb X$ and $\mb Y$, satisfy the conditions of Proposition \ref{prop:main_FUP_prop}. Indeed, $\mb X, \mb Y$ have dimension $< 2$, so certainly $|\mb X| = |\mb Y| = 0$. 

Let $\mc X_k, \mc Y_k \subset \Z_N^2$, $N = M^k$, be the $k$th Cantor iterates. Then $\mc X_k \subset X_{N}$, $\mc Y_k \subset Y_N$, where $X_N$ and $Y_N$ are obtained by discretizing $\mb X, \mb Y$ as in \S\ref{sec:disc_of_fractal_set}. By Proposition \ref{prop:main_FUP_prop}, for $N$ large enough there is no $f: \Z_N^2\to \C$ with $\supp f \subset X_N$ and $\supp \hat f \subset Y_N$. Thus for $k$ large enough, there is no $f$ with $\supp f \subset \mc X_k$ and $\supp \hat f \subset Y_k$. For this $k$, 
\begin{align*}
	\|1_{\mc Y_k} \mc F1_{\mc X_k}\|_{2 \to 2} < 1
\end{align*}
and so by submultiplicativity we conclude 
\begin{align*}
	\|1_{\mc Y_k} \mc F1_{\mc X_k}\|_{2 \to 2} \lesssim M^{-k\beta}
\end{align*}
for some $\beta > 0$. 
\end{proof}

\section{Proof of the main Lemma}\label{sec:main_lemma}
In 1965 Lang conjectured \cite{Lang} that if $C$ is an irreducible algebraic curve in $\C^{\times n}$ with infinitely many cyclotomic points---that is, points $(z_1, \ldots, z_n) \in C$ all of which are roots of unity---then $C$ is a translate of a subgroup of $\C^{\times n}$ by a root of unity \cite{GranvilleRudnick}. 

The key ingredient in proving Lemma \ref{lem:main_lemma} is the following Theorem of Ruppert \cite{Ruppert}*{Corollary 5} and Beukers \& Smyth \cite{BeukersSmyth}*{Theorem 4.1}, which can be viewed as a sharp quantitative form of Lang's conjecture in two dimensions.
\begin{theorem}[Ruppert 1993, Beukers-Smyth 2002]\label{thm:beukers_smyth}
Let
\begin{equation}\label{eq:deg_leq_D_poly}
	F(z,w) = \sum_{0 \leq k,l \leq D} a_{kl}z^k w^l{}
\end{equation}
be a polynomial in $\C[z,w]$ with degree at most $D$ in $z$, $w$ seperately. Then $F$ has either at most $22 D^2$ cyclotomic points, or infinitely many. In the latter case $F$ has an irreducible factor 
\begin{equation}\label{eq:poly_good_form}
	z^a w^b - \zeta\quad \text{ or }\quad z^a - \zeta w^b
\end{equation}
 for some root of unity $\zeta$ and coprime integers $a,b$.
\end{theorem}
We note that $z^aw^b - \zeta$ or $z^a - \zeta w^b$ is only irreducible if $a$ and $b$ are coprime integers, which is why that is part of the conclusion. In their paper Beukers \& Smyth actually prove significantly more, they give an algorithm to compute this factor. 
The approach is to find seven polynomials $F_1, \ldots, F_7$ so that every cyclotomic root of $F$ is also a root of some $F_j$, and then apply Bezout's inequality to bound their pairwise intersection, see \S\ref{sec:beukers-smyth-argument} for sketch. In what follows
\begin{equation}
	\deg F = \max_{a_{kl} \neq 0} \max(|k|, |l|),\quad F(z,w) = \sum_{k,l} a_{kl} z^kw^l
\end{equation}
so that (\ref{eq:deg_leq_D_poly}) is the general form of a polynomial with degree $\leq D$.

Recall that we can embed $\T^2 \to \C^{\times 2}$ via
\begin{align*}
	(x, y) \mapsto (e^{2\pi i x}, e^{2\pi i y}). 
\end{align*}
The cyclotomic points in $\C^{\times 2}$ are precisely the image of $(\Q/\Z)^2$. 
For $F(z,w)$ a polynomial, we let 
\begin{align}
	\ZZ(F) &= \{(x,y) \in \T^2\, :\, F(e^{2\pi i x}, e^{2\pi i y}) = 0\},\nonumber \\ 
	\ZZ_N(F) &= \{(x,y) \in \Z_N^2\, :\, F(e^{\frac{2\pi i}{N} x}, e^{\frac{2\pi i}{N} y}) = 0\}. \label{eq:Z_N_eq}
\end{align}
If we view $\Z_N^2$ as the subgroup of $\T^2$ given by
\begin{equation*}
	\Z_N^2 \cong \T_N^2 = \left\{\Bigl(\frac{x}{N}, \frac{y}{N}\Bigr)\in \T^2\, |\, x, y \in \Z\right\} 	
\end{equation*} 
then $\ZZ_N(F) = \ZZ(F) \cap \T_N^2$.  
We say that a polynomial $F$ of the form (\ref{eq:poly_good_form}) \textit{cuts out a line} because 
\begin{equation*}
	\ZZ(F) = \{(x,y) \in \T^2\, :\, ax+by = c\} \text{ or } \ZZ(F) = \{(x,y) \in \T^2\, :\, ax-by=c\},
\end{equation*}
with $a,b \geq 0$ integers and $c \in \Q$. If $c = \frac{c'}{N}$, $c' \in \Z$, then we say $\ell$ cuts out a line in $\Z_N^2$. Conversely, suppose
\begin{align*}
	\ell = \{(x,y) \in \Z_N^2\, :\, ax+by=c \pmod N\}
\end{align*}
is an irreducible line. By Lemma \ref{lem:lines_gen_one_elem}, $a, b$ can be taken as coprime integers. Then
\begin{align*}
	\ell = \ZZ_N(P_{\ell}),\quad P_{\ell}(z,w) = \begin{cases}
		z^aw^b - e^{2\pi i c/N} & a,b \geq 0 \\ 
		z^a - e^{2\pi i c/N}w^{|b|} & a \geq 0,\ b < 0
	\end{cases}
\end{align*}
and $P_{\ell}$ is an irreducible polynomial with $\deg P_{\ell} \leq 2\|\ell\|$. 
Theorem \ref{thm:beukers_smyth} is related to Lemma \ref{lem:main_lemma} because functions $g: \Z_N^2 \to \C$ with $\supp \hat g \subset [0,D] \times [0,D]$ have values given by polynomials at cyclotomic points: 
\begin{equation*}
	g(x,y) = \frac{1}{N}\sum_{0 \leq k, l \leq D} \hat g(k,l) z^k w^l,\quad z = e^{\frac{2\pi i}{N} x}, w = e^{\frac{2\pi i}{N} y}.
\end{equation*}

Lemma \ref{lem:main_lemma} is a quick consequence of the following. We don't try to optimize the constant $200$.

\begin{lemma}\label{lem:polynomials_through_sets}
Let $S \subset \Z_N^2$ be an arbitrary nonempty set. Then there is a polynomial $F^*$ with $\deg F^* < 200|S|^{1/2} - 1$ so that $S - \ZZ_N(F^*)$ is nonempty and lies on an irreducible line $\ell$ with $\|\ell\| \leq 200|S|^{1/2}$. 
\end{lemma}
We prove the slightly awkward bound $\deg F^* < 200|S|^{1/2} - 1$ in order to make the proof and application of Lemma \ref{lem:main_lemma} cleaner. 
Before proving this Lemma, it is helpful to consider how it could fail to be true. Consider a quadratic polynomial
\begin{equation*}
	G(z,w) = a+bz+cw+fzw+dz^2+ew^2
\end{equation*}
which does not cut out a line (e.g. $G$ is not of the form $z = w^2$). 
Theorem \ref{thm:beukers_smyth} says that $|\ZZ_N(G)| \leq 44$ for all $N$ (the quadratic polynomial $G$ cannot pass through many cyclotomic points). Ignoring this fact for a moment, it turns out that if for some $G, N$, $|\ZZ_N(G)| > 1200^2$, Lemma \ref{lem:polynomials_through_sets} would fail. 

Let $S = \ZZ_N(G)$. Suppose $F^*$ is a polynomial of degree $\leq 200|S|^{1/2}$ such that $S - \ZZ_N(F^*)$ is nonempty and lies on a line $\ell$ with $\|\ell\| \leq 200|S|^{1/2}$. 
The polynomial $G$ cannot be a component of $F^*$, because that would mean $S \subset \ZZ_N(F^*)$. So by Bezout's inequality (\ref{thm:bezout_theorem}), 
\begin{align*}
	|\ZZ_N(F^*) \cap S| \leq 2 \deg F^* \leq 400|S|^{1/2}.
\end{align*}
If $\ell$ is a line with $\|\ell\| \leq 200|S|^{1/2}$, then
\begin{align*}
	|\ell \cap S| = |\ZZ_N(P_{\ell}) \cap S| \leq 2 \deg P_{\ell} \leq 4\|\ell\| \leq 800|S|^{1/2}
\end{align*}
again by Bezout's inequality. 
Thus if $S - \ZZ_N(F^*)$ lies on such a line $\ell$,
\begin{align*}
	|S| - 800|S|^{1/2} \leq 400|S|^{1/2} \Rightarrow |S| \leq 1200^2
\end{align*}
as claimed. 
Before proving Lemma \ref{lem:polynomials_through_sets} we need another lemma. 
\begin{lemma}\label{lem:poly_through_set_fact}
For every nonempty set $S \subset \C^2$, $D = \lfloor |S|^{1/2} \rfloor$, there is a nonzero polynomial $F(z,w) = \sum_{0 \leq k,l \leq D} a_{kl}z^kw^l$ vanishing on $S$.
\end{lemma}
\begin{proof}
Consider the linear map taking 
\begin{equation*}
(a_{kl})_{0 \leq k,l\leq D} \mapsto \left( \sum_{kl} a_{kl} z^kw^l\, :\, (z,w) \in S\right).
\end{equation*}
If $(D+1)^2 > |S|$ then by rank nullity this has a nontrivial kernel, which is our desired polynomial $F$. Thus we may take $D = \lfloor |S|^{1/2} \rfloor$
\end{proof}
The proof of Lemma \ref{lem:polynomials_through_sets} involves four cases, see Figure \ref{fig:poly_through_set_pf_vis}.
\begin{figure}
\includegraphics[width=0.8\linewidth]{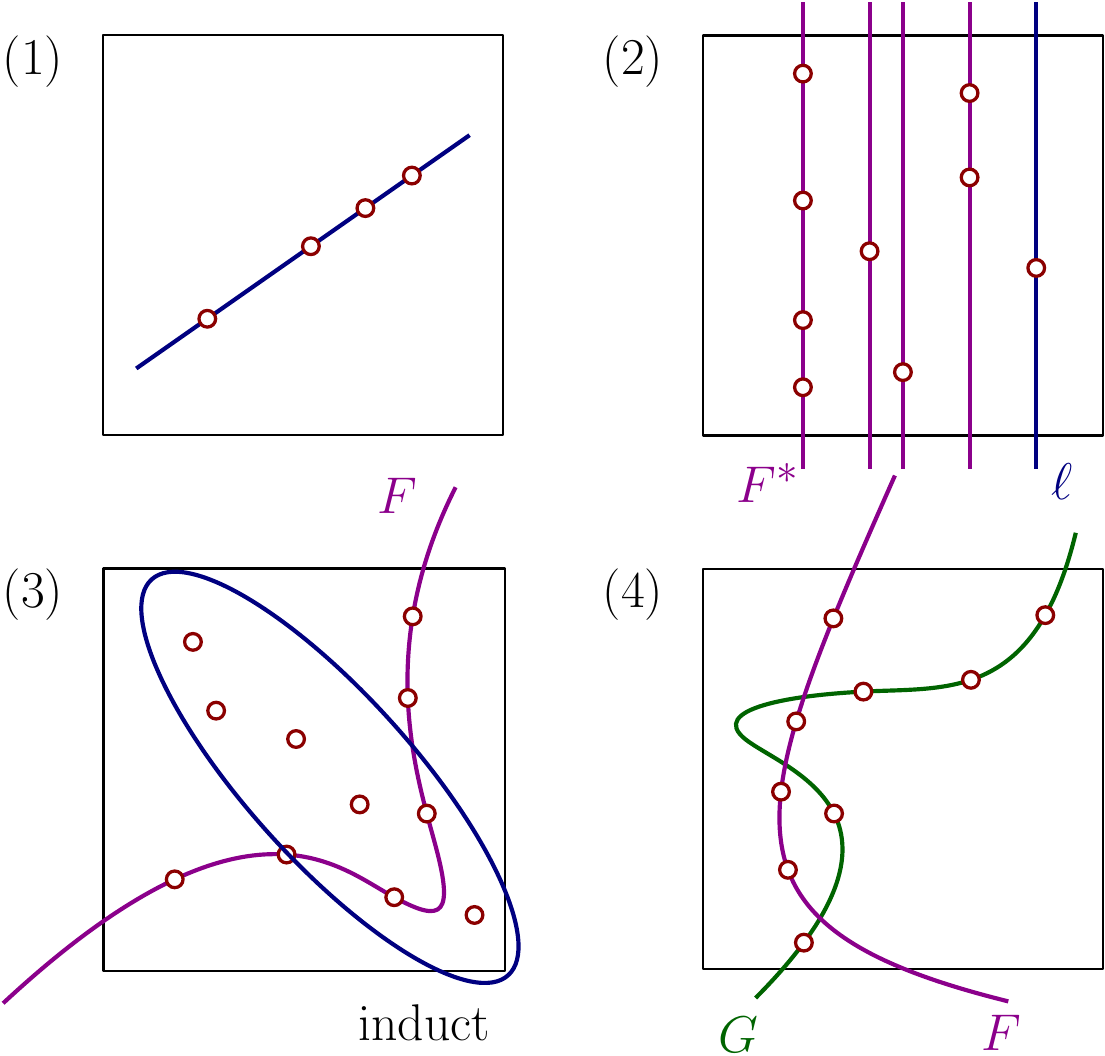}
\caption{Cases in the proof of Lemma \ref{lem:polynomials_through_sets}}.
\label{fig:poly_through_set_pf_vis}
\end{figure}

\begin{proof}[Proof of Lemma \ref{lem:polynomials_through_sets}]
We give a recursive algorithm to find our polynomial $F^*$. Mathematically this is phrased as induction on the size of $S$. For ease of presentation we prove we can take $\deg F^* \leq 200|S|^{1/2}$, but the same argument can be optimized to give $\deg F^* \leq 198|S|^{1/2}$ yielding the claim $\deg F^* < 200|S|^{1/2} - 1$. 

Let $F$ be a polynomial of minimal degree $D$ with $S \subset \ZZ_N(F)$. We have $D \leq |S|^{1/2}$ by Lemma \ref{lem:poly_through_set_fact}. If there are several such polynomials, choose one with the minimal number of irreducible factors.  

\par{\textbf{Case 1}: $F$ cuts out a line $\ell$.} In this case 
\begin{align*}
	\ZZ_N(F) = \Bigl\{(x,y) \in \Z_N^2\, :\, \frac{ax}{N}+\frac{by}{N}=c\Bigr\},\ c \in \Q.
\end{align*}
Because $S$ is nonempty there is some $(x_0, y_0) \in S$ with $c = (ax_0+by_0)/N$. So $F$ cuts out a line $\ell$ in $\Z_N^2$, and 
\begin{align*}
	\ZZ_N(F) &= \ell = \{(x,y) \in \Z_N^2\, :\, ax+by=ax_0+by_0\}, \\ 
	\|\ell\| &\leq \deg F \leq |S|^{1/2}. 
\end{align*}
Thus we are already done---we may take $F^* = 1$, and $S$ already lies on a desired line $\ell$. 

\par{\textbf{Case 2}: $|S| \leq 200$.} Let $S = \{(x_k, y_k) \in \Z_N^2\}$, and $\{x_1, \ldots, x_m\}$ be the distinct $x$-coordinates appearing in $S$. If $m = 1$, we are in case 1. Otherwise, set 
\begin{equation*}
	F^* = (z - e^{2\pi i x_1/N})\dots(z-e^{2\pi i x_{m-1}/N}).
\end{equation*}
Then $\deg F^* < 200 \leq 200|S|^{1/2}$, and 
\begin{align*}
	S - \ZZ_N(F^*) \subset \{x = x_m\}
\end{align*}
lies on a line. 

\par{\textbf{Case 3}: $F$ is irreducible but does not cut out a line}. In this case, $|S| \leq 22 D^2$ by Theorem \ref{thm:beukers_smyth}. Because $|S| \geq 200$, $D \geq 4$. Choose a curve $G$ of degree $D-1$ passing through at least 
\begin{equation*}
	(D-1)^2 \geq \frac{1}{22}|S| (1-1/D)^2 \geq \frac{|S|}{40}
\end{equation*}
points of $S$. Let $A = S\cap \ZZ_N(G)$. Notice $S - A$ is nonempty by the minimality of $D$. Now apply the inductive hypothesis to find a polynomial $H$ passing through all but one line of $S - A$ with $\deg H \leq 200|S-A|^{1/2}$, and set $F^* = GH$. We have 
\begin{equation*}
	\deg G \leq 200\sqrt{|S|(1-1/40)} + D-1 \leq |S|^{1/2}(198 + 1) \leq 200|S|^{1/2}
\end{equation*}
as needed. 

\par{\textbf{Case 4}: $F$ is reducible.} 
Let $F = GH$ where neither $G$ nor $H$ are scalars and $|\ZZ_N(G)| \leq |\ZZ_N(H)|$. 
Let $T = \ZZ_N(G) - \ZZ_N(H)$. Because $\deg H \leq \deg F$ and $H$ has fewer irreducible factors than $F$, $S \not\subset \ZZ_N(H)$, so $T$ is nonempty. 
Using the inductive hypothesis we may find a polynomial $P$ passing through all but one line of the set $T = \ZZ_N(G) - \ZZ_N(H)$ (notice $T$ is nonempty by minimality of the number of irreducible factors). We have $|T| \leq |S|/2$. Set $F^* = HP$. We have
\begin{equation*}
	\deg F^* \leq 200|T|^{1/2} + \deg H \leq 200(|S|/2)^{1/2} + |S|^{1/2} \leq  143|S|^{1/2}
\end{equation*}
as needed. 
\end{proof}

Now we prove Lemma \ref{lem:main_lemma}.

\begin{proof}[Proof of Lemma \ref{lem:main_lemma}]
Let $f: \Z_N^2 \to \C$ have $\supp f = S$. Let $R = \lfloor 200|S|^{1/2}\rfloor$. 
By Lemma \ref{lem:polynomials_through_sets} let, 
\begin{align*}
	F^* &= \sum_{0 \leq k,l < R} a_{k,l} z^kw^l \in \C[z,w] \\
	\ell &= \{(x,y) \in \Z_N^2\, :\, ax+by = c\},\quad \max(|a|, |b|) \leq R
\end{align*}
be such that
\begin{align*}
	A = S - \ZZ_N(F^*)
\end{align*}
is nonempty and lies on $\ell$.
Let $h: \Z_N^2\to \C$ be defined by 
\begin{align*}
h(x,y) = \frac{1}{N}F(e^{\frac{2\pi i}{N} x}, e^{\frac{2\pi i}{N} y}),\quad 
	\hat h(k,l) = \begin{cases}
		a_{k,l} & 0 \leq k,l < R, \\ 
		0 & \text{else}.
	\end{cases}
\end{align*}
Thus $hf$ is nonzero and supported in $\ell$. Also, 
\begin{equation*}
	\supp \widehat{hf} = \supp \hat h * \hat f \subset \supp (1_{[0,R)^2} * 1_{\supp \hat f}) = \N_{R}(\supp \hat f).
\end{equation*}
Setting $g := hf$ we are done.
\end{proof}

\begin{remark}
In order to obtain Theorem \ref{thm:2D_FUP}, it would suffice to replace Theorem \ref{thm:beukers_smyth} with a quantitatively weaker version which says that for $F(z,w)$ a degree $D$ irreducible polynomial not cutting out a line, 
\begin{align*}
	\#\{(\zeta_1, \zeta_2)\, :\, F(\zeta_1, \zeta_2) = 0\} \lesssim_{\varepsilon} D^{2+\varepsilon},\quad \zeta_1, \zeta_2 \text{ cyclotomic}
\end{align*}
for all $\varepsilon > 0$. 
\end{remark}

\section{Loose ends}\label{sec:loose_ends}

\subsection{Submultiplicativity}\label{sec:submul_proof}
We prove the submultiplicativity estimate (\ref{eq:submul}) in two dimensions. 
\begin{proof}
We first recall how Dyatlov proves submultiplicativity for discrete 1D cantor sets in \cite{DyatlovFUPSurvey}*{Lemma 4.6}. The Fourier transform $\mc F: \Z_{N_1N_2} \to \Z_{N_1N_2}$ can be realized as follows. We realize $L^2(\Z_{N_1N_2})$ as $L^2(\mathrm{Mat}_{N_1 \times N_2})$. In this basis, 
\begin{align*}
	\mc F =  \mc F_{col} D \mc F_{row}
\end{align*}
where 
\begin{alignat*}{2}
	(\mc F_{row}U)_{pb} &= \frac{1}{\sqrt{N_2}} \sum_{q=0}^{N_2} e^{-\frac{2\pi i}{N_2} bq} U_{pq}\quad &&\text{ applies the Fourier transform to each row}, \\ 
	(\mc F_{col}U)_{pb} &= \frac{1}{\sqrt{N_1}} \sum_{q=0}^{N_1} e^{-\frac{2\pi i}{N_1} pq} U_{qb}\quad &&\text{ applies the Fourier transform to each column}, \\ 
	(DU)_{pb} &= e^{-\frac{2\pi i}{N} pb} U_{pb}\quad &&\text{ applies a phase shift to each entry}.
\end{alignat*} 
Abstractly, $L^2(\Z_{N_1N_2}) = L^2(\Z_{N_1}) \otimes L^2(\Z_{N_2})$. But 
\begin{equation*}
	\mc F_{\Z_{N_1}} \otimes \mc F_{\Z_{N_2}} = \mc F_{\Z_{N_1} \times \Z_{N_2}} \neq \mc F_{\Z_{N_1N_2}}.
\end{equation*}
The phase shift operator $D$ corrects this issue. We can write 
\begin{align}
	\mc F_{row} = \mathrm{Id} \otimes \mc F_{\Z_{N_2}},\quad \mc F_{col} = \mc F_{\Z_{N_1}} \otimes \mathrm{Id} \nonumber\\ 
	\mc F = (\mc F_{\Z_{N_1}} \otimes \mathrm{Id})\circ D \circ (\mathrm{Id} \otimes \mc F_{\Z_{N_2}}) \label{eq:fourier_comp_eq}. 
\end{align}
In the notation of tensor products, if $N_1 = M^{k}$, $N_2 = M^r$, then 
\begin{align*}
	1_{\mc X_{k+r}} = 1_{\mc X_k} \otimes 1_{\mc X_r},\quad 1_{\mc Y_{k+r}} = 1_{\mc Y_k} \otimes 1_{\mc Y_r}. 
\end{align*}
Because $1_{\mc X_{k+r}}$, $1_{\mc Y_{k+r}}$ commute with $D$, 
\begin{align*}
	1_{\mc Y_{k+r}} \mc F_{k+r} 1_{\mc X_{k+r}} &= (1_{\mc Y_k} \otimes 1_{\mc Y_r})\circ  (\mc F_{\Z_{N_1}} \otimes \mathrm{Id})\circ D \circ (\mathrm{Id} \otimes \mc F_{\Z_{N_2}})\circ (1_{\mc X_k} \otimes 1_{\mc X_r}) \\ 
	&= (1_{\mc Y_k} \mc F_{\Z_{N_1}} \otimes 1_{\mc Y_r})\circ D \circ (1_{\mc X_k} \otimes \mc F_{\Z_{N_2}}1_{\mc X_r}) \\ 
	&= (1_{\mc Y_k} \mc F_{\Z_{N_1}} 1_{\mc X_k} \otimes 1_{\mc Y_r})\circ D \circ (1_{\mc X_k} \otimes 1_{\mc Y_r} \mc F_{\Z_{N_2}}1_{\mc X_r}).
\end{align*}
It follows from the above that 
\begin{align*}
	\|1_{\mc Y_{k+r}} \mc F_{k+r} 1_{\mc X_{k+r}}\|_{2\to 2} \leq \|1_{\mc Y_k} \mc F_{\Z_{N_1}} 1_{\mc X_k}\|_{2\to 2} \|1_{\mc Y_r} \mc F_{\Z_{N_2}} 1_{\mc X_r}\|_{2\to 2} 
\end{align*}
as desired. Written in this way, it is easy to see that the submultiplicativity estimate extends to two dimensions. We have the equation 
\begin{align*}
	\mc F_{\Z_{N_1N_2}^2} = (\mc F_{\Z_{N_1}^2} \otimes \mathrm{Id}) \circ D \circ (\mathrm{Id} \otimes \mc F_{\Z_{N_2}^2})
\end{align*}
where $D$ is a multiplication operator (indeed, this can be seen from writing $\Z_{N_1N_2}^2$ as a product of two copies of $\Z_{N_1N_2}$ and tensoring Equation (\ref{eq:fourier_comp_eq}) with itself) and the rest of the proof goes through verbatim. 
\end{proof}

\subsection{Theorem \ref{thm:2D_FUP} is sharp}\label{sec:thm_sharp}
Suppose $\mc A, \mc B$ are alphabets generating fractal sets $\mb X, \mb Y \subset \T^2$ with 
\begin{align*}
	&\R\mb v + \mb p \subset \mb X\text{ and }\R \mb v^{\perp} + \mb q \subset \mb Y,\\
	&\mb v = (a,b),\ \mb v^{\perp} = (-b, a),\quad \text{$a$ and $b$ coprime integers}.
\end{align*}
 We show $\mc A, \mc B$ do not obey a FUP. This amounts to showing that for infinitely many $k$ (in fact, for all $k$), there exists $f: \Z_N^2\to \C$ with 
\begin{align*}
	\supp f \subset \mc X_k,\quad \supp \hat f \subset \mc Y_k
\end{align*}
where $\mc X_k$, $\mc Y_k$ are defined in Equation (\ref{eq:X_k_cantor_defn}). 

\par{\textbf{Case 1:} $a = 0$ or $b = 0$.}
Assume $(a, b) = (0, 1)$. Then $\mb X$ contains a vertical line and $\mb Y$ contains a horizontal line. It follows that $\mc A$ contains some vertical line $\{x= x_0\}$, $x_0 \in \Z_M$, and $\mc B$ contains a horizontal line $\{y = y_0\}$. Let 
\begin{align*}
	x^{(k)} &= x_0 + Mx_0 + \dots + M^{k-1} x_0,\quad \{(x^{(k)}, y)\, :\, y \in \Z_N\} \subset \mc X_k\\
	y^{(k)} &= y_0 + My_0 + \dots + M^{k-1} y_0,\quad \{(x, y^{(k)})\, :\, y \in \Z_N\} \subset \mc Y_k.
\end{align*}
We have 
\begin{align*}
	\mc F N^{-1/2}1_{x = 0}(\xi, \eta) = N^{-3/2}\sum_{y \in \Z_N} e^{-\frac{2\pi i}{N} y\eta} = N^{-1/2}1_{\eta = 0}
\end{align*}
so 
\begin{align*}
	f = N^{-1/2}e^{2\pi i y^{(k)}} 1_{x = x^{(k)}}
\end{align*}
satisfies 
\begin{align*}
	\supp f = \{x = x^{(k)}\} \subset \mc X_k,\quad \supp \hat f = \{y = y^{(k)}\} \subset \mc Y_k
\end{align*}
as needed.

\par{\textbf{Case 2:} $a,b \neq 0$. }

In this case we claim 
\begin{equation}\label{eq:in_Xk_cond}
\begin{split}
	\mc X_k = \{(x,y) \in \Z_N^2\, :\, (x/M, (x+1)/M) \times (y/M, (y+1)/M) \cap \mb X \neq \emptyset\}, \\ 
	\mc Y_k = \{(x,y) \in \Z_N^2\, :\, (x/M, (x+1)/M) \times (y/M, (y+1)/M) \cap \mb Y \neq \emptyset\}.
\end{split}
\end{equation}
It is clear that if $(x/M, (x+1)/M) \times (y/M, (y+1)/M) \cap \mb X \neq \emptyset$ then $(x,y) \in \mc X_k$. For the other direction, we first note that $(0, 1)^2 \cap \mb X \neq \emptyset$---the only way for this to fail is if $\mc A$ lies on one of the horizontal or vertical lines $x = 0$, $x = M-1$, $y = 0$, $y = M-1$ in which case we are back in case 1. Now if $(x,y) \in \mc X_k$, then $(x, x+1) \times (y, y+1) \cap \mb X \neq \emptyset$ by self similarity of $\mb X$. 

Now, assume without loss of generality that $a,b$ are coprime. We will show that for all $k$, there exists $\mb p^{(k)}, \mb q^{(k)} \in \Z_{M^k}^2$ so that 
\begin{equation}\label{eq:lines_contained}
	\Z\mb v + \mb p^{(k)} \subset \mc X_k \text{ and } \Z \mb v^{\perp} + \mb q^{(k)} \subset \mc Y_k.
\end{equation}
We show it just for $\mc X_k$. By \ref{eq:in_Xk_cond}, we would like to choose $\mb p^{(k)} = (p^{(k)}_1, p^{(k)}_2)$ so that for all $t \in \Z$, 
\begin{align*}
	p^{(k)}_1 + ta &< M^kp_1 + (t+\varepsilon)a < p^{(k)}_1 + ta+1, \\ 
	p^{(k)}_2 + tb &< M^kp_2 + (t+\varepsilon)b < p^{(k)}_2 + tb+1,
\end{align*}
for some small $\varepsilon$. Rearranging, this amounts to
\begin{equation}\label{eq:line_offset_eqs}
\begin{split}
	0 &< (M^kp_1 - p^{(k)}_1) + \varepsilon a < 1, \\ 
	0 &< (M^kp_2 - p^{(k)}_2) + \varepsilon b < 1.
\end{split} 
\end{equation}
To make this true, we select $p^{(k)}_1, p^{(k)}_2$ to be integers so that
\begin{align*}
	M^kp_1 - p^{(k)}_1 &\in \begin{cases}
		[0,1) & \text{ if } a > 0 \\
		(0, 1] & \text{ if } a < 0
	\end{cases},\\
	M^kp_2 - p^{(k)}_2 &\in \begin{cases}
		[0,1) & \text{ if } b > 0 \\
		(0, 1] & \text{ if } b < 0
	\end{cases}.
\end{align*}
In each of these cases (\ref{eq:line_offset_eqs}) will hold, which yields (\ref{eq:lines_contained}). Now, we have 
\begin{align*}
	\mc F 1_{\Z\mb v}(\mb \xi) = \frac{1}{\sqrt{N}}\sum_{t \in \Z_N} e^{-\frac{2\pi i }{N} t\mb v \cdot \mb \xi} = 1_{\mb v \cdot \mb \xi= 0} = 1_{\Z\mb v^{\perp}}
\end{align*}
by Lemma \ref{lem:lines_gen_one_elem}. Thus with $T_{\mb a} f = f(\cdot - \mb a)$ the translation operator, we see that 
\begin{align*}
	f = \mc F^{-1} T_{\mb q^{(k)}} \mc F T_{\mb p^{(k)}} 1_{\Z\mb v}
\end{align*}
satisfies 
\begin{align*}
	\supp f \subset \Z \mb v + \mb p^{(k)} \subset \mc X_k,\quad \supp \hat f \subset \Z \mb v^{\perp} + \mb q^{(k)} \subset \mc Y_k
\end{align*}
contradicting a fractal uncertainty principle. 

\subsection{Proof of Proposition \ref{prop:lines_in_cantor_set}}\label{sec:pf_lines_in_cantor_prop}
Let $\mc A \subset \Z_M^2$ be an alphabet and $\mb X \subset \T^2$ the Cantor set it generates. Let $\mb A \subset \T^2$ be the drawing of $\mc A$.

First we show that a line $\R \mb v + \mb p$ lies on $\mb X$ if and only if $\R \mb v + M^k\mb p$ lies on $\mb A$ for all $k \geq 0$. Recall that
\begin{align*}
	\mb x \in \mb X \text{ if and only if } M^k \mb x \in \mb A \text{ for all } k \geq 0. 	
\end{align*} 
 Now, suppose $\R \mb v + \mb p \subset \mb X$. Then
\begin{align*}
	t\mb v + \mb p \in \mb X \Rightarrow (M^k t)\mb v + M^k\mb p \in \mb X
\end{align*}
so rescaling, $\R \mb v + M^k \mb p \subset \mb X \subset \mb A$. In the reverse direction, suppose $\R \mb v + M^k\mb p \subset \mb A$ for all $k$. Then
\begin{equation*}
	M^k(t\mb v + \mb p) \in \mb A \text{ for all } k \Rightarrow t \mb v + \mb p \in \mb X
\end{equation*}
as needed. Also, by Lemma \ref{lem:denseness_fact}, if $\mb v = (a, b)$ and $\max(|a|, |b|) > M$ then $\R \mb v + \mb p \not\subset \mb A$ for any $\mb p$.

\subsubsection*{More discussion on a procedure for checking lines}
Suppose $\ell = \R \mb v + \mb p$. If $\mb v$ is a multiple of $(1, 0)$ then $\ell$ is a horizontal line, and $\mb X$ can only contain a horizontal line if $\mc A$ does. 

Otherwise, let $\mb v = (a,b)$ with $a, b$ coprime integers, $b \neq 0$. Assume $a, b$ are fixed and $\max(|a|, |b|) \leq M$. 
There is some $\mb p' = (p_0, 0) \in \ell$, so $\ell = \R \mb v + (p_0, 0)$. 
We will turn the question around and consider the closed set
\begin{align*}
	\mb S_{\mb v} = \{s \in \T\, :\, \R \mb v + (s, 0) \subset \mb A\}. 
\end{align*}
The only possible boundary points are those for which $\R \mb v + (s, 0)$ intersects a point of the form $(j/M, k / M) \in \T^2$. If $t\mb v + (s, 0) = (j/M, k/M)$ then we can compute $s$ as
\begin{align*}
	s = \frac{jb-ka+ Mr}{Mb} \text{ for some } r \in \Z,\text{ so } s = \frac{c}{Mb} \text{ for some $c \in \Z$}.
\end{align*}
Now we can write $\mb S_{\mb v}$ as a union of intervals,
\begin{align*}
	S_{\mb v} &= \left\{s_j \in \{0,\ldots, Mb - 1\}\, :\, \Bigl[\frac{s_j}{Mb}, \frac{s_j+1}{Mb}\Bigr] \subset \mb S_{\mb v}\right\}, \\ 
	\mb S_{\mb v} &= \bigcup_{s_j\in S_{\mb v}} \Bigl[\frac{s_j}{Mb}, \frac{s_j+1}{Mb}\Bigr].
\end{align*}
Given the alphabet $\mb A$ and $\mb v = (a, b)$, one can efficiently compute the finite set $S_{\mb v} \subset \Z_{Mb}$. It is then a combinatorial question whether or not there exists $x \in \T$ so that $M^k x \in \mb S_{\mb v}$ for all $k$. It would be interesting to find an algorithm to answer this question.

\subsection{FUP implies spectral gap for bakers maps}\label{sec:proof_application_bakers}
We would like to show that the results in \cite{DyatlovJinBakersMaps}*{\S2} hold for two dimensional bakers maps, in particular Proposition \ref{prop:dyatlov_jin_FUP_specgap} (\cite{DyatlovJinBakersMaps}*{Proposition 2.6}). We prove here that \cite{DyatlovJinBakersMaps}*{Proposition 2.3} holds in 2D. The deduction of Proposition 2.4 from Proposition 2.3 is the same in 2D, and the proofs of Proposition 2.5 and 2.6 go through verbatim. 

In what follows we use the $\ell^{\infty}$ distance on $\T^2$ as in (\ref{eq:linf_dist_T2}). Let $\Phi$ be the expanding map
\begin{align*}
	&\Phi = \Phi_{M, \mc A} : \bigsqcup_{\mb a \in \mc A} \Bigl(\frac{a_1}{M}, \frac{a_1+1}{M}\Bigr)\times\Bigl(\frac{a_2}{M}, \frac{a_2+1}{M}\Bigr) \to (0,1)^2,\\
	&\Phi(x,y) = (Mx - a_1, My - a_2),\quad (x, y) \in \Bigl(\frac{a_1}{M}, \frac{a_1+1}{M}\Bigr) \times \Bigl(\frac{a_2}{M}, \frac{a_2+1}{M}\Bigr).
\end{align*}
For each $\varphi: \T^2 \to \R$ define
\begin{align*}
	\varphi_N \in \ell^2(\Z_N^2),\quad \varphi_N(\mb j) = \varphi(\mb j / N). 
\end{align*}
The function $\varphi_N$ defines a multiplication operator as well as a Fourier multiplier $\varphi_N^{\mc F} = \mc F_N^* \varphi_N \mc F_N$. 
\begin{proposition}[Propagation of singularities]\label{prop:propagation_singularities}
Assume that $\varphi, \psi: \T^2 \to [0,1]$ and for some $c > 0$, $0 \leq \rho < 1$, 
\begin{equation}
	d(\Phi(\supp \psi \cap \Phi^{-1}(\supp \chi)), \supp \varphi) \geq cN^{-\rho}. 
\end{equation}
Then
\begin{align}
	\|\varphi_N B_N \psi_N\|_{2 \to 2} = \mc O(N^{-\infty}),\label{eq:prop_sing_mult} \\ 
	\|\varphi_N^{\mc F} B_N \psi_N^{\mc F}\|_{2 \to 2} = \mc O(N^{-\infty}), \label{eq:prop_sing_fourier_mult}
\end{align}
where $\mc O(N^{-\infty})$ means decay faster than any polynomial, with constants depending only on $c, \rho, \chi$. In particular, these hold when 
\begin{equation}
	d(\supp \psi, \Phi^{-1}(\supp \varphi)) \geq cN^{-\rho}.
\end{equation}
\end{proposition}
The proof is almost identical to that in \cite{DyatlovJinBakersMaps}.
\begin{proof}
We have 
\begin{align*}
	\varphi_N B_N \psi_N u(\mb j) &= \sum_{\mb a \in \mc A} \sum_{\stackrel{0 \leq k_1, k_2 \leq N/M - 1}{\mb k = (k_1, k_2)}} A^{\mb a}_{\mb j\mb k} u\Bigl(\mb k + \mb a\frac{N}{M}\Bigr),  \\
	A^{\mb a}_{\mb j \mb k} &= \frac{M}{N^2} \varphi\Bigl(\frac{\mb j}{N}\Bigr) \exp\Bigl(\frac{2\pi i}{M} \mb a \cdot \mb j\Bigr) \chi\Bigl(\mb k \frac{M}{N}\Bigr) \psi\Bigl(\frac{\mb k}{N} + \frac{\mb a}{M}\Bigr) \tilde A_{\mb j \mb k}, \\ 
	\tilde A_{\mb j \mb k} &= \sum_{\stackrel{0 \leq m_1, m_2 \leq N/M - 1}{\mb m = (m_1, m_2)}} \exp\Bigl(\frac{2\pi i}{N} \mb m \cdot (\mb j - \mb k M)\Bigr) \chi\Bigl(\mb m \frac{M}{N}\Bigr).
\end{align*}
We can write 
\begin{align*}
	\tilde A_{\mb j \mb k} = \sum_{\mb m \in \Z_N^2} \exp\Bigl(\frac{2\pi i}{N} \mb b\cdot \mb m\Bigr)\chi_1\Bigl(\frac{\mb m}{N}\Bigr),\quad \mb b L= \mb j - \mb k M,\quad \chi_1(\mb x) = \chi(M\mb x). 
\end{align*}
We have $A^{\mb a}_{\mb j\mb k} = 0$ unless 
\begin{equation*}
	\frac{\mb j}{N} \in \supp \varphi,\quad \frac{\mb k}{N} + \frac{\mb a}{M} \in \supp \psi,\quad \mb k \frac{M}{N} = \Phi(\frac{\mb k}{N} + \frac{\mb a}{M}) \in \supp \chi. 
\end{equation*}
It follows that $d\Bigl(\frac{\mb b}{N}, 0\Bigr) \geq cN^{-\rho}$, so by the method of nonstationary phase \cite{DyatlovJinBakersMaps}*{Lemma 2.2}, we see $\max_{\mb a, \mb j, \mb k} |A^{\mb a}_{\mb j \mb k}| = \mc O(N^{-\infty})$ and (\ref{eq:prop_sing_mult}) follows. Equation (\ref{eq:prop_sing_fourier_mult}) is a consequence, as 
\begin{align*}
	\psi_N^{\mc F} B_N \varphi_N^{\mc F} = \mc F_N^*\overline{(\varphi_N B_N \psi_N)}^* \mc F_N
\end{align*}
and the Fourier transform is unitary. 
\end{proof}

\section{Acknowledgements}
Thanks to Larry Guth and Peter Sarnak for helpful discussions. Many thanks to Semyon Dyatlov for detailed and helpful comments, and for useful conversations along the way. 

\appendix
\section{Sketch of the proof of Theorem \ref{thm:beukers_smyth}}\label{sec:beukers-smyth-argument}

In this section we will try to illustrate the main ideas of Beukers \& Smyth's proof of Theorem \ref{thm:beukers_smyth} as directly as possible. In what follows the degree of a polynomial is
\begin{align*}
	\deg F = \max_{a_{kl} \neq 0} (k+l),\quad F(z,w) = \sum_{k,l} a_{kl} z^kw^l \in \C[z,w]
\end{align*}
which is different from \S\ref{sec:main_lemma}. We will use the notation $Z_N(F) \subset \Z_N^2$ as in (\ref{eq:Z_N_eq}).

\subsection{Bezout's inequality}

We first state Bezout's theorem. 
\begin{theorem}\label{thm:bezout_theorem}
Let $F, G \in \C[z,w]$ be coprime irreducible polynomials with degrees $D, E$ which are not multiples of each other, 
\begin{align*}
	F = \sum_{0 \leq k+l \leq D} a_{kl} z^kw^l,\quad G = \sum_{0 \leq k+l \leq E} b_{kl} z^kw^l. 
\end{align*}
Then 
\begin{align*}
	|\{(z,w) \in \C^2\, :\, F(z,w) = G(z,w) = 0\}| \leq DE.
\end{align*}
If intersections are taken in $\C \mathbb{P}^2$ and counted with multiplicity, then this is an equality. 
\end{theorem}
We denote by $\V(F)$, $\V(G) \subset \C^2$ the zero sets of $F$ and $G$. Then Bezout's inequality can be written 
\begin{equation}\label{eq:bezout_ineq_eq}
	|\V(F) \cap \V(G)| \leq DE.
\end{equation}

\subsection{Setup for Theorem \ref{thm:beukers_smyth}}

To prove Theorem \ref{thm:beukers_smyth}, it is more convenient to work with Laurent polynomials $F \in \C[z,w,z^{-1},w^{-1}]$. Like polynomials in two variables, Laurent polynomials in two variables also enjoy unique factorization up to units and satisfy a version of Bezout's inequality. From this perspective, the factors $z^a - \zeta w^b$ can be written as $z^aw^{-b} - \zeta$, so we can just look for factors of the form $z^aw^b - \zeta$. 

Beukers \& Smyth make the following reduction. For $F = \sum_{kl} a_{kl}z^kw^l \in \C[z,w,z^{-1},w^{-1}]$ a Laurent polynomial, let $\mc L(F)$ be the sublattice of $\Z^2$ generated by 
\begin{align*}
	\{(k,l) - (k', l')\, :\, a_{kl}, a_{k',l'} \neq 0\}. 
\end{align*}
Notice that if $F = z^aw^b - \zeta$, then $\mc L(F) = \Z(a,b)$ has rank one. If $F = z^a - \zeta w^b$, then $\mc L(F) = \Z(a, -b)$. More generally, if $\mc L(F)$ has rank one then $F$ can be written as a function of $z^aw^b$ and one can reduce to the one variable case. If $\mc L(F)$ has rank two but is not all of $\Z^2$, one can change variables within the class of Laurent polynomials to reduce to the case where $\mc L(F) = \Z^2$. Rather than fully explain this, we will just prove Theorem \ref{thm:beukers_smyth} in the case where $F$ is a genuine polynomial and $\mc L(F) = \Z^2$. 

Here is part of Lemma 1 from \cite{BeukersSmyth}. 
\begin{lemma}\label{lem:root_of_unity_conj}
If $\zeta$ is a root of unity, then it is Galois conjugate to exactly one of $-\zeta, \zeta^2, -\zeta^2$. 
\end{lemma}

Now we partially prove a lemma covering the relevant portions of \cite{BeukersSmyth}*{\S3}. We follow them directly. 

\begin{lemma}\label{lem:seven_polys}
Let $F \in \C[z,w]$ be an irreducible polynomial with $\mc L(F) = \Z^2$. Then there are seven other polynomials $F_1, \ldots, F_7$ none of which have $F$ as a component, and such that if $(z,w)$ is a cyclotomic point ($z^N = w^N = 1$ for some $N$) with $F(z,w) = 0$, then $F_j(z,w) = 0$ for some $1 \leq j \leq 7$. We may take 
\begin{align*}
	\deg F_1 &= \deg F_2 = \deg F_3 = \deg F, \\
	\deg F_4 &= \deg F_5 = \deg F_6 = \deg F_7 = 2\deg F. 
\end{align*}
\end{lemma}

It follows directly from Bezout's inequality (\ref{eq:bezout_ineq_eq}) that 
\begin{align*}\label{eq:beukerssmythconclusion}
	\ZZ_N(F) &\subset \bigcup_{j=1}^7 \ZZ_N(F) \cap \ZZ_N(F_j) \quad \text{ for all } N,\\ 
	|\ZZ_N(F)| &\leq 3D^2 + 8D^2 = 11D^2\quad \text{ for all } N.
\end{align*}
\begin{remark}
In Theorem \ref{thm:beukers_smyth} we state the bound $22D^2$ rather than $11D^2$ because we allow terms of the form $z^Dw^D$ which has degree $2D$. The bound is $22D^2$ rather than $11(2D)^2 = 44D^2$ because the \textit{Newton polytope} of $F$ has volume $\leq D^2$, so \cite{BeukersSmyth}*{Theorem 4.1} gives the sharper bound of $22D^2$. 
\end{remark}

\subsection{Proof sketch of some special cases of Lemma \ref{lem:seven_polys}}

In the proof we split into cases depending on whether or not $F$ can be defined over an abelian extension of $\Q$. The hardest case is when $F$ is defined in some nontrivial abelian extension of $\Q$---there are a few subcases involved. We prove Lemma \ref{lem:seven_polys} in the two easier cases where $F$ has coefficients in $\Q$, and where $F$ is not defined over any abelian extension.

First, multiply $F$ be a constant so one of its coefficients are rational. 
\par{\textbf{Case 1:} $F \in \Q[z,w]$.}
We take 
\begin{align*}
	F_1 = F(-z, w),\quad F_2 = F(z, -w),\quad F_3 = F(-z, -w) \\ 
	F_4 = F(z^2, w^2),\quad F_5(-z^2, w^2),\quad F_6(z^2, -w^2),\quad F_7(-z^2,-w^2).
\end{align*}
We must show that if $F(z,w) = 0$ is a cyclotomic point, $F_j(z,w) = 0$ for some $j$. Let $\zeta$ be a root of unity and $z = \zeta^a$, $w = \zeta^b$, $a,b$ coprime. Then $f(\zeta) = F(\zeta^a, \zeta^b)$ is a polynomial in $\zeta$ with rational coefficients. Thus every conjugate of $\zeta$ is also a root of $f$. By Lemma \ref{lem:root_of_unity_conj} exactly one of $\{-\zeta, \zeta^2, -\zeta^2\}$ is conjugate to $\zeta$, so one of 
\begin{align*}
	(-\zeta^a, \zeta^b), (\zeta^a, -\zeta^b), (-\zeta^a, -\zeta^b), (\zeta^{2a}, \zeta^{2b}), (-\zeta^{2a}, \zeta^{2b}), (\zeta^{2a}, -\zeta^{2b}), (-\zeta^{2a}, -\zeta^{2b})
\end{align*}
is also a zero of $F$ as needed. It remains to show that $F$ is not a component of any $F_j$. Because they arise from a linear change of variables of $F$, $F_1,F_2,F_3$ are irreducible. 
If $F_1$ is a linear multiple of $F$, then all nonzero $a_{kl}$ must have the same parity for $k$. Thus $\mc L(F)$ would span a proper sublattice of $\Z^2$ contradicting our assumption. Similar arguments show that $F_2$ and $F_3$ are not linear multiples of $F$, and because they have the same degree, $F$ is not a component. If $F$ were a component of $F_4$ then $F(z^2,w^2) = F(z,w)G(z,w)$, so $F(z^2,w^2) = F_1(z,w) G(-z,w)$, and $F_1$ is a component of $F_4$ as well. An analagous argument shows $F_2, F_3$ are components as well. This would imply that $\deg F_4 \geq \deg F F_1F_2F_3 \geq 4D$ using the fact that $F_1, F_2, F_3, F_4$ are all distinct irreducibles. But $\deg F_4 = 2D$, a contradiction. A similar argument shows $F$ is not a factor of $F_5, F_6, F_7$. 

\par{\textbf{Case 2:} the coefficients of $F$ do not lie in any abelian extension of $\Q$}. 
This case is easier. Let $\sigma \in \Gal(\C / \Q^{ab})$ be an automorphism of $\C$ which fixes $\Q^{ab}$ and does not fix the coefficients of $F$. Here $\Q^{ab}$ is the maximal abelian extension of $\Q$, which is the composite of all the cyclotomic extensions $\Q[e^{2\pi i /N}]$. Let 
\begin{align*}
	F^{\sigma} = \sum_{kl} \sigma(a_{kl}) z^kw^l. 
\end{align*}
For $z,w$ a cyclotomic root of $F$, $\sigma(z) = z$ and $\sigma(w) = w$, so 
\begin{align*}
	F^{\sigma}(z,w) = \sigma(F(z,w)) = 0. 
\end{align*}
Thus the cyclotomic points of $F$ are contained in $V(F) \cap V(F^{\sigma})$. But $F^{\sigma}$ is not a multiple of $F$, because some coefficient of $F$ (the rational one) is fixed by $\sigma$ and another must be different. Thus $V(F) \cap V(F^{\sigma}) \leq D^2$. 

\section{Higher dimensions and continuous FUP}\label{sec:higher_dim_continuous}

\subsection{Results from a new method}

It seems difficult to use the ideas in the present paper to prove a discrete FUP in $d\geq 3$ dimensions. We would need a higher dimensional analogue of Theorem \ref{thm:beukers_smyth} with very strong bounds that are currently unavailable. 

However, after this work was completed the author \cite{Cohen2023} proved a fractal uncertainty principle for sets $\mb X \subset \R^d$ that avoid lines in a quantitative sense called \textit{line porosity}. The core of the latter paper involves constructing plurisubharmonic functions, and the methods are completely different from those used here---there is no arithmetic input. Using the new work we can prove the following higher dimensional result for discrete Cantor sets.

\begin{theorem}\label{thm:discrete_FUP_higher_dim}
Suppose $\mc A, \mc B \subsetneq \Z_M^d$ are alphabets with drawings $\mb X, \mb Y \subset \T^d$. If $\mb Y$ does not contain any lines, then $\mc X_k, \mc Y_k$ satisfy 
\begin{equation*}
	\| 1_{\mc Y_k} \mc F 1_{\mc X_k} \|_{2\to 2} \lesssim M^{-k\beta}
\end{equation*}
for some $\beta > 0$. 
\end{theorem}

The more recent work has a few advantages. We don't need self-similarity, the result applies in any dimension, and most importantly, we move from the model setting $\Z_N^d$ to the physically relevant domain $\R^d$. 

On the other hand, Theorem \ref{thm:2D_FUP} gives a precise condition involving pairs of orthogonal lines which is currently unavailable in the continuous setting: Theorem \ref{thm:discrete_FUP_higher_dim} needs one of the Cantor sets to avoid all lines. It is an interesting challenge to improve the main result of \cite{Cohen2023} so the condition involves pairs of orthogonal subspaces.

\subsection{Statement of higher dimensional continuous FUP}

For $\mb x \in \R^d$, let $\B_{R}(\mb x)$ be the radius $R$ ball about $\mb x$. 

\begin{definition} 
Let $\nu \leq 1/3$.
\begin{itemize}
	\item A set $\mb X \subset \R^d$ is \textit{$\nu$-porous on balls} from scales $\alpha_0$ to $\alpha_1$ if for every ball $\B$ of diameter $\alpha_0 < R < \alpha_1$ there is some $\mb x \in  \B$ such that $\B_{\nu R}(\mb x) \cap \mb X = \emptyset$.

	\item A set $\mb X$ is \textit{$\nu$-porous on lines from scales $\alpha_0$ to $\alpha_1$} if for all line segments $\tau$ with length $\alpha_0 < R < \alpha_1$, there is some $\mb x \in \tau$ such that $\B_{\nu R}(\mb x) \cap \mb X = \emptyset$.
\end{itemize}
\end{definition}
We are ready to state the main theorem of \cite{Cohen2023}.

\begin{theorem}\label{thm:continuous_FUP_higher_dim}
Let $\nu>0$ and assume that
\begin{itemize}
    \item $\mb X \subset [-1, 1]^d$ is $\nu$-porous on balls from scales $h$ to $1$, and 
    \item $\mb Y \subset [-h^{-1},h^{-1}]^d$ is $\nu$-porous on lines from scales $1$ to $h^{-1}$. 
\end{itemize}
Then there exist $\beta, C > 0$ depending only on $\nu$ and $d$ such that for all $f \in L^2(\R^d)$ 
\begin{equation}\label{eq:higher_dim_FUP_estimate}
    \supp \hat f \subset \mb Y\, \Longrightarrow\, \| f 1_{\mb X}\|_2 \leq C\, h^{\beta} \| f \|_2.
\end{equation}
\end{theorem}

To prove Theorem \ref{thm:discrete_FUP_higher_dim} we first show that the drawing of a Cantor set avoiding lines is porous on lines, and then prove a discrete FUP using continuous FUP. 

\subsection{Line porosity for self similar Cantor sets}

In this section $\mb x \in [0,1]^d$ denotes a point in $\R^d$ and $\mbbar x \in \T^d$ denotes the image in the torus. Similarly for sets $\mb Y \subset [0,1]^d$ and $\mbbar Y \subset \T^d$. 

\begin{definition}
Let $\mbbar X \subsetneq \T^d$ be a closed set. We say $\mbbar X$ is a self-similar Cantor set at level $M$ if $M\cdot \mbbar X = \mbbar X$, where 
\begin{equation*}
	M\cdot \mbbar X = \{M\mbbar x\, :\, \mbbar x \in \mbbar X\}.
\end{equation*}
\end{definition}
In particular, if $\mc X_k$ is a sequence of Cantor sets in $\Z_{M^k}$, then the drawing $\mbbar X \subset \T^d$ is a self-similar Cantor set. 

We first prove that if a Cantor set does not contain any lines, then it also does not contain any line segments. By a \textit{line in $\T^d$}, we mean an irreducible one dimensional closed coset. By a line segment $\bar \tau \subset \T^d$, we mean the image of a line segment in $\R^d$. 

\begin{lemma}\label{lem:cantor_not_contain_line_segments}
Let $\mbbar Y \subset \T^d$ be a self-similar Cantor set which contains no lines. Then $\mbbar Y$ also does not contain any line segments $\bar \tau$. 
\end{lemma}
\begin{proof}
Suppose by way of contradiction that $\bar \tau \subset \T^d$ is a line segment with $\bar \tau \subset \mbbar Y$. 
Let $\bar \tau$ point in direction $\mbhat v \in S^{d-1}$, and let 
\begin{equation*}
    C_{\mbhat v} = \mathrm{cl}\, \{t\bar{\mbhat v}\, :\, t \in \R\} \subset \T^d
\end{equation*}
be the closure of the geodesic based at the origin and pointing in direction $\mbhat v$. The set $C_{\mbhat v}$ is a closed subgroup which contains at least one torus line. Choose $\mb x_0$ in the interior of $\bar{\tau}$.
Select a subsequence $\{k_j\}_{j\geq 0}$ such that $M^{k_j} \mbbar x_0 \to \mbbar x_0' \in \T^d$. For any $t \in \R$, 
\begin{equation*}
    M^{k_j} (\mbbar x_0 + M^{-k_j} t \bar{\mbhat v}) \to \mbbar x_0' + t \bar{\mbhat v} \in \mbbar x_0' + C_{\mbhat v}. 
\end{equation*}
For large enough $j$, $M^{k_j}(\mbbar x_0 + M^{-k_j} t \bar{\mbhat v})\in  \mbbar Y$, and because $\mbbar Y$ is closed, we see $\mbbar x_0' + C_{\mbhat v} \subset \mbbar Y$ contradicting our assumption.
\end{proof}

We prove if a Cantor set does not contain lines then it is porous on lines. 
\begin{lemma}\label{lem:cantor_set_line_porous}
Suppose that $\mbbar Y \subset \T^d$ is a self-similar Cantor set which does not contain any lines. Then for some $\nu > 0$, $\mb Y\subset [0,1]^d$ is $\nu$-porous on lines from scales $0$ to $1$.
\end{lemma}
\begin{proof}
Let $\mbbar Y \subset \T^d$ be a Cantor set which does not contain any lines. 
We show by a compactness argument that for some $c_0 > 0$, every line segment $\bar \tau$ with length $1$ has some $\mbbar x \in \bar \tau$ such that $d(\mbbar x, \mbbar Y) \geq c_0$. Suppose by way of contradiction that this is not the case. Then there is a sequence $\bar \tau_j$ of unit line segments such that $\max_{\mbbar x \in \bar \tau_j} d(\mbbar x, \mbbar Y) \leq c_j$ where $c_j \to 0$. The space of unit line segments in $\T^d$ is compact, so there is some line segment $\bar \tau$ which is a limit of these, and it follows that $\bar \tau \subset \mbbar Y$ contradicting Lemma \ref{lem:cantor_not_contain_line_segments}.

Now let $\tau \subset \R^d$ be a line segment of length $0 < R < 1$. We would like to show there is some $\mb x \in \tau$ such that $d(\mb x, \mb Y) \geq \nu R$. The torus metric is stronger than the ambient $\R^d$ metric, so it suffices to show that there is some $\mbbar x \in \mbbar \tau$ such that $d(\mbbar x, \mbbar \tau) \geq \nu R$. Let $j\geq 0$ be the smallest integer so that $M^j R \geq 1$. Because $M^j \cdot \bar \tau$ is a line segment with length $\geq 1$, there is some $\mbbar x \in \bar \tau$ such that $d(M^j \mbbar x, \mb Y) \geq c_0$. So by self similarity $d(\mbbar x, \mb Y) \geq M^{-j} c_0 \geq \frac{c_0}{M} R$ and $\mb Y$ is $\nu$-porous on lines from scales $0$ to $1$ with $\nu = \frac{c_0}{M}$. 
\end{proof}

\subsection{Proof of Theorem \ref{thm:discrete_FUP_higher_dim}}

We roughly follow Dyatlov and Jin's argument in \cite{DyatlovJinDolgopyat}*{Proposition 5.8}. We state a general proposition which allows us to prove discrete fractal uncertainty from continuous fractal uncertainty. 
We will need the locally constant property from Fourier analysis, which we explain in a certain form now. Construct a $w \in C^{\infty}(\R^d)$ by setting $\widehat w$ to be a smooth bump function with $\widehat w = 1$ on $B_1$ and $\supp \widehat w \subset B_2$. Then 
\begin{equation*}
    |w(\mb x)| \lesssim_{m,d} \langle \mb x \rangle^{-m}\quad \text{for all $m \geq 0$}. 
\end{equation*}
Moreover, if $f \in L^2(\R^d)$ is a function with $\supp \hat f \subset B_N$ then 
\begin{equation*}
    f = f * w_N,\quad w_N(\mb x) = N^d w(N\mb x). 
\end{equation*}
In particular we have the pointwise bound 
\begin{equation}\label{eq:locally_const_pointwise_bd}
    |f(\mb x)| \lesssim N^{d/2}\| f(\cdot) w(N(\cdot - \mb x)) \|_2.
\end{equation}
Let $\mc X, \mc Y \subset \Z_N^d$ be sets. Let 
\begin{equation}\label{eq:point_sets_for_mod}
\begin{split}
    X &= N^{-1}\cdot \{\mb x \in \{0,\ldots,N-1\}^d\, :\, \mbbar x \in \mc X\}, \\ 
    Y &= \{\mb y \in \{0,\ldots,N-1\}^d\, :\, \mbbar y \in \mc Y\}.
\end{split}
\end{equation}
Here is the main proposition connecting discrete and continuous FUP. 
\begin{proposition}\label{prop:estimate_2to2_discrete}
Let $\mc X, \mc Y \subset \Z_N^d$ and $X, Y \subset \R^d$ be as above. For any $10/N < r < 1/10$ and $m > 0$ we have
\begin{equation}\label{eq:discrete_2to2_norm_estimate}
    \| 1_{\mc X}\, \mc F\, 1_{\mc Y} \|_{2\to 2} \lesssim_{d,m} \| 1_{X + B_r}\, \mc F\, 1_{Y + B_{1/4}} \|_{2\to 2} + (Nr)^{-m}.
\end{equation}
\end{proposition}
\begin{proof}
Let $u \in L^2(\Z_N^d)$ have $\supp \hat u \subset \mc Y_k$. We will construct an auxiliary function $f \in L^2(\R^d)$ based on $u$. 

Let $\chi \in C_0^{\infty}(\R^d)$ be a bump function supported in $B_{1/4}$. We can design $\chi$ so that 
\begin{align}
 |\chi^{\vee}(\mb x)| &\geq 1 && \text{for $\mb x \in [-10,10]^d$}, \\ 
 \| \chi \|_2 &\leq C_d. 
\end{align}
Let $f$ be given by 
\begin{equation*}
    \hat f(\xi) = \sum_{\mb \xi' \in \{0,\ldots, N-1\}^d} \hat u(\xi') \chi(\xi - \xi'). 
\end{equation*}
We have 
\begin{equation*}
    \| f \|_2^2 = \| u \|_2^2 \| \chi \|_2^2 \lesssim \| u\|_2^2.
\end{equation*}
Notice that for $\mb x \in X$,
\begin{equation*}
    f(\mb x) = N^{d/2}\, \chi^{\vee}(\mb x)\, u(\overline{N \mb x})
\end{equation*}
so
\begin{equation*}
    \| u\, 1_{\mc X_k} \|_{L^2(\Z_N^d)}^2 \lesssim N^{-d} \sum_{\mb x \in X} |f(\mb x)|^2. 
\end{equation*}
If we let
\begin{equation*}
    \tilde w(\mb x) = \Bigl(\sum_{\mb x' \in X} |w(N(\mb x - \mb x'))|^2\Bigr)^{1/2}
\end{equation*}
by (\ref{eq:locally_const_pointwise_bd}), $|f(\mb x)|^2 \lesssim N^d \| w(N(x-x')) f \|_2^2$, so summing over $\mb x \in X$ we find
\begin{align*}
    \sum_{\mb x \in X} |f(\mb x)|^2 &\lesssim N^d \| f\, \tilde w\|_2^2, \\ 
    \| u 1_{\mc X_k}\|_2^2 &\lesssim \| f \tilde w \|_2^2. 
\end{align*}
Using the fact that $X$ is an $N^{-1}$-separated set, 
\begin{align*}
    |\tilde w(\mb x)|^2 &= \sum_{\mb x' \in X} |w(N(\mb x - \mb x'))|^2 \lesssim_m \sum_{\mb x' \in X} (1+ N |\mb x - \mb x'|)^{-m}\\
    &\lesssim \sum_{N^{-1} \leq 2^j \leq 10} (1+N\, 2^j)^{-m}\, |X \cap B_{2^j}(\mb x)| \lesssim \sum_{2^j \geq \max(N^{-1}, d(\mb x, X))} (N2^j)^{d-m}\\
    &\lesssim (1+N\,d(\mb x, X))^{d-m}
\end{align*}
for $m$ large enough. Thus for any $r > N^{-1}$ and $m\geq 0$,
\begin{equation*}
    |\tilde w(\mb x)| \lesssim_{m,d} 1_{X + B_r}(\mb x) + (Nr)^{-m}.
\end{equation*}
Because $\supp \hat f \subset Y + B_{1/4}$, 
\begin{align*}
    \| u 1_{\mc X_k} \|_2 &\lesssim \| f\, 1_{X+B_r} \|_2 + (Nr)^{-m} \| f \|_2 \\  &\lesssim (\| 1_{X + B_r}\, \mc F\, 1_{Y+B_{1/4}}\|_{2\to 2} + (Nr)^{-m}) \| u \|_2
\end{align*}
giving (\ref{eq:discrete_2to2_norm_estimate}).
\end{proof}

Now we prove the FUP for arithmetic Cantor sets that avoid lines. 

\begin{proof}[Proof of Theorem \ref{thm:discrete_FUP_higher_dim}]
Let $\mc X_k$ and $\mc Y_k$ be a sequence of Cantor iterates such that the drawing $\mbbar Y \subset \T^d$ does not contain any lines. Let $N = M^k$. Let $X_k \subset [0,1]^d$ and $Y_k \subset [0,N]^d$ be the corresponding point sets as in (\ref{eq:point_sets_for_mod}). By choosing $r = N^{\varepsilon - 1}$ in Proposition \ref{prop:estimate_2to2_discrete}, we have for any $\varepsilon > 0$ the estimate
\begin{equation*}
    \| 1_{\mc X_k}\, \mc F\, 1_{\mc Y_k} \|_{2\to 2} \lesssim \| 1_{X_k + B_{r}}\, \mc F\, 1_{Y_k + B_{1/4}}\|_{2\to 2} + N^{-\varepsilon}.
\end{equation*}
Letting $\mb X, \mb Y \subset [0,1]^d$ be the drawings of these Cantor sets, we have
\begin{align}
    X_k \subset \mb X + [-N^{-1}, N^{-1}]^d,\quad Y_k \subset N\cdot \mb Y + [-1,1].
\end{align}
Thus
\begin{itemize}
    \item The set $\mb X$ is $\nu$-porous on balls from scales $0$ to $1$. So $X_k + B_{N^{\varepsilon-1}}$ is $\nu$-porous on balls from scales $2N^{\varepsilon-1}$ to $1$. 
    \item By Lemma \ref{lem:cantor_set_line_porous}, the set $\mb Y$ is $\nu$-porous on lines from scales $0$ to $1$. So the set $Y_k + B_{1/4}$ is $\nu$-porous on lines from scales $\sqrt{d}/4$ to $N$.
\end{itemize}
In the above, the value of $\nu$ changes from line to line.
Split up $[-N-1,N+1]^d$ into a disjoint union of $\lesssim N^{\varepsilon d}$ many cubes $Q \in \mc Q$ that have side length $N^{1-\varepsilon}$. By Theorem \ref{thm:continuous_FUP_higher_dim}, there is $\beta = \beta(\nu, d) > 0$ so that
\begin{equation*}
    \| 1_{X_k + B_r}\, \mc F\, 1_{(Y_k + B_{1/4}) \cap Q} \|_{2\to 2} \lesssim N^{-(1-\varepsilon)\beta}.
\end{equation*}
Summing this over all the boxes $Q\in \mc Q$, we have
\begin{equation*}
    \| 1_{X_k + B_r}\, \mc F\, 1_{Y_k + B_{1/4}} \|_{2\to 2} \lesssim N^{-\beta+\varepsilon (d+\beta)}. 
\end{equation*}
Choose $\varepsilon > 0$ small enough that the exponent is negative and apply Proposition \ref{prop:estimate_2to2_discrete} to obtain
\begin{equation*}
    \| 1_{\mc X_k}\, \mc F \, 1_{\mc Y_k} \|_{2\to 2} \leq CN^{-\beta'}
\end{equation*}
for some $\beta' > 0$. 
\end{proof}

\bibliographystyle{agsm}
\bibliography{references}

\end{document}